\documentclass[11pt]{amsart}
\usepackage[top=25truemm,bottom=25truemm,left=25truemm,right=25truemm]{geometry}

\usepackage{amssymb}
\usepackage{amsfonts}
\usepackage{mathrsfs}
\usepackage[all]{xy}
\usepackage{bm}
\usepackage{amsmath}
\usepackage{stmaryrd}
\usepackage{url}

\if@nosecthm
    \newtheorem{theorem}{Theorem}
\else
    \newtheorem{theorem}{Theorem}[section]
\fi

    \newtheorem{corollary}[theorem]{Corollary}
    
    \newtheorem{definition}[theorem]{Definition}
    \newtheorem{example}[theorem]{Example}
    \newtheorem{lemma}[theorem]{Lemma}

    \newtheorem{proposition}[theorem]{Proposition}
    \newtheorem{question}[theorem]{Question}
    \newtheorem{remark}[theorem]{Remark}
\newcommand{\DS}{\displaystyle}

\newcommand{\SC}{\scriptstyle}
\newcommand{\SSC}{\scriptscriptstyle}

\DeclareMathOperator{\Cok}{Coker}
\DeclareMathOperator{\End}{End}

\DeclareMathOperator{\Frob}{Frob}
\DeclareMathOperator{\Gal}{Gal}
\DeclareMathOperator{\GL}{GL}
\DeclareMathOperator{\Gr}{Gr}
\DeclareMathOperator{\Hom}{Hom}
\DeclareMathOperator{\HP}{\mathcal{HP}}
\DeclareMathOperator{\id}{id}

\DeclareMathOperator{\Ker}{Ker}

\DeclareMathOperator{\Lie}{Lie}

\DeclareMathOperator{\ord}{ord}

\DeclareMathOperator{\rk}{rk}
\DeclareMathOperator{\Spec}{Spec}
\DeclareMathOperator{\Spf}{Spf}
\DeclareMathOperator{\Sp}{Sp}

\DeclareMathOperator{\TI}{\mathcal{TI}}
\DeclareMathOperator{\Tor}{Tor}
\DeclareMathOperator{\Tr}{tr}
\DeclareMathOperator{\W}{\mathcal{W}}
\DeclareMathOperator{\Wt}{wt}

\newcommand{\ra}{\rightarrow}
\newcommand{\rai}{\overset{\sim}{\ra}}

\newcommand{\hra}{\hookrightarrow}
\newcommand{\thra}{\twoheadrightarrow}
\newcommand{\lra}{\longrightarrow}

\newcommand{\lrai}{\overset{\sim}{\lra}}
\newcommand{\Lra}{\Longrightarrow}

\newcommand{\dpl}{{\mathchoice{\mbox{\rm (\hspace{-0.15em}(}}
                              {\mbox{\rm (\hspace{-0.15em}(}}
                              {\mbox{\scriptsize\rm (\hspace{-0.15em}(}}
                              {\mbox{\tiny\rm (\hspace{-0.15em}(}}}}
\newcommand{\dpr}{{\mathchoice{\mbox{\rm )\hspace{-0.15em})}}
                              {\mbox{\rm )\hspace{-0.15em})}}
                              {\mbox{\scriptsize\rm )\hspace{-0.15em})}}
                              {\mbox{\tiny\rm )\hspace{-0.15em})}}}}

\newcommand{\an}{{\rm an}}
\newcommand{\dR}{{\rm dR}}

\newcommand{\rig}{{\rm rig}}
\newcommand{\sep}{{\rm sep}}
\renewcommand{\ss}{{\rm ss}}
\newcommand{\tame}{{\rm tame}}

\newcommand{\ur}{{\rm ur}}
\newcommand{\wild}{{\rm wild}}

\newcommand{\zb}{\llbracket z \rrbracket	}     %   [[z]]
\newcommand{\zp}{\dpl z \dpr}     %   ((z))
\newcommand{\zf}{\frac{1}{z}}       %    1/z

\newcommand{\zz}{z-\zeta}          %z-zeta
\newcommand{\zzb}{\llbracket z-\zeta \rrbracket	}  
\newcommand{\zzp}{\dpl z-\zeta \dpr}
\newcommand{\zzf}{\frac{1}{z-\zeta}}

\newcommand{\izb}{\llbracket z_\infty \rrbracket	}     %   [[z]]
\newcommand{\izp}{\dpl z_\infty \dpr}     %   ((z))
\newcommand{\sig}{\sigma}
\renewcommand{\phi}{\varphi}
\newcommand{\lam}{\lambda}

\newcommand{\hsig}{\wh{\sig}}

\renewcommand{\ast}{^{\SC*}}
\newcommand{\wc}{\check}
\newcommand{\es}{\enspace}

\newcommand{\mal}{^{\SSC\times}}
\newcommand{\inv}{^{\SSC-1}}

\newcommand{\col}{\colon}
\newcommand{\ul}[1]{{\underline{#1\hspace{-0.18em}}\hspace{0.18em}}}
\newcommand{\ol}[1]{{\overline{#1}}}
\newcommand{\wh}[1]{{\hat{#1}}}
\newcommand{\wt}[1]{{\tilde{#1}}}

\newcommand{\ot}{\otimes}
\newcommand{\op}{\oplus}

\usepackage{ifthen}

\newcommand{\ilim}[1][]{\ifthenelse{\equal{#1}{}}% falls Argument leer
{\DS \lim_{\longleftarrow}}%                         verwende niedrige Version
{\DS \lim_{\underset{#1}{\longleftarrow}}}%  sonst:  verwende Argument
}

\newcommand{\dlim}[1][]{\ifthenelse{\equal{#1}{}}% falls Argument leer
{\DS \lim_{\longrightarrow}}%                        verwende niedrige Version
{\DS \lim_{\underset{#1}{\longrightarrow}}}% sonst:  verwende Argument
}

\makeatletter %\rmB => \mathrm{B}, \sB => \mathscr{B}, \bB => \mathbb{B}, \cB => \mathcal{B} \rmb => \mathrm{b}, \frb=> \mathfrak{b}, \frB=>\mathfrak{B},\bsB =>\boldsymbol{B},\bsb=>\boldsymbol{b}, \bfB =>\mathbf{B},\bfb=>\mathbf{b}
\@tempcnta\z@
\loop\ifnum\@tempcnta<26
\advance\@tempcnta\@ne
\expandafter\edef\csname rm\@Alph\@tempcnta\endcsname{\noexpand\mathrm{\@Alph\@tempcnta}}
\expandafter\edef\csname s\@Alph\@tempcnta\endcsname{\noexpand\mathscr{\@Alph\@tempcnta}}
\expandafter\edef\csname b\@Alph\@tempcnta\endcsname{\noexpand\mathbb{\@Alph\@tempcnta}}
\expandafter\edef\csname c\@Alph\@tempcnta\endcsname{\noexpand\mathcal{\@Alph\@tempcnta}}
\expandafter\edef\csname rm\@alph\@tempcnta\endcsname{\noexpand\mathrm{\@alph\@tempcnta}}
\expandafter\edef\csname fr\@alph\@tempcnta\endcsname{\noexpand\mathfrak{\@alph\@tempcnta}}
\expandafter\edef\csname fr\@Alph\@tempcnta\endcsname{\noexpand\mathfrak{\@Alph\@tempcnta}}
\expandafter\edef\csname bs\@Alph\@tempcnta\endcsname{\noexpand\boldsymbol{\@Alph\@tempcnta}}
\expandafter\edef\csname bs\@alph\@tempcnta\endcsname{\noexpand\boldsymbol{\@alph\@tempcnta}}
\expandafter\edef\csname bf\@Alph\@tempcnta\endcsname{\noexpand\mathbf{\@Alph\@tempcnta}}
\expandafter\edef\csname bf\@alph\@tempcnta\endcsname{\noexpand\mathbf{\@alph\@tempcnta}}
\repeat

 \makeatletter
           
    \@addtoreset{equation}{section}
  \makeatother

\begin{document}

%%%%%%%%%%%%%%%%%%%%%%%%%%%%%%%%%%%%%%%%%%%%%%%%%%%%%%%%%%%%%%%%%%%%%%%%%%%%%%%%%%%%%%%%%
%% 基本情報

\author{Yoshiaki Okumura}
\title{Congruences of Galois representations attached to effective $A$-motives over global function fields}
\date{}

\maketitle
\thispagestyle{empty}

%%%%%%%%%%%%  Abstract %%%%%%%%%%%%%%%%%%%%%%%%%%%%%%%%%%%%%%%%%%%%%%%%%%%%%%%%%
\begin{abstract}
This article investigates congruences of $\frp$-adic representations arising from effective $A$-motives defined over a global function field $K$. 
We give a criterion for two congruent $\frp$-adic representations coming from strongly semi-stable effective  $A$-motives to be isomorphic up to semi-simplification when restricted to decomposition groups of suitable places of $K$.
This is a function field analog of Ozeki-Taguchi's criterion for $\ell$-adic representations of number fields.
Motivated by a non-existence conjecture on abelian varieties over number fields stated by  Rasmussen and Tamagawa, we also show that there exist no strongly semi-stable effective $A$-motives with some constrained.
\end{abstract}

%%%%%%%%%%%%  脚注, 目次, 2010MSC, Keyword %%%%%%%%%%%%%%%%%%%%%%%%%%%%%%%%%%%

\pagestyle{myheadings}
\markboth{Y.\ Okumura}{Congruent Galois representations of $A$-motives}

\renewcommand{\thefootnote}{\fnsymbol{footnote}}
\renewcommand{\thefootnote}{\arabic{footnote}}
\renewcommand{\thefootnote}{\fnsymbol{footnote}}
\renewcommand{\thefootnote}{\arabic{footnote}}

\setcounter{tocdepth}{1}
\tableofcontents  %目次

%2020 MSC

%11F80  	Galois representations
%11R58  	Arithmetic theory of algebraic function fields 
%11G09  	Drinfel'd modules; higher-dimensional motives
%13A35  	Characteristic $p$ methods (Frobenius endomorphism) and reduction to characteristic $p$

%%%%%%%%%%%%%%%%%%%%%%%%%%%%%%%%%%%%%%%%%%%%%%%%%%%%%%%%%%%%%%%%%%%%%%%%%%%%%%%%
%%%%%      	以下、本文
%%%%%%%%%%%%%%%%%%%%%%%%%%%%%%%%%%%%%%%%%%%%%%%%%%%%%%%%%%%%%%%%%%%

%%%%%%%%%%%%%%%%%%%%%%%%%%%%%%%%%%%%%%%%%%%%%%%%%%%%%%%%%%%%%%%%%%%%%%%%%%%%%%%%%%%%%%%%%%%%%%%%%%%%%%%%%%%%%%%%%%%%%%%%%%%%%%%

%%%%%%%%%%%%%%%%%%%%%%%%%%%%%%%%%%%%%%%%%%%%%%%%%%%%%%%%%%%%%%%%%%%%%%%%%%%%%%%%%%%%%
%
%   1. Introduction 
%
%%%%%%%%%%%%%%%%%%%%%%%%%%%%%%%%%%%%%%%%%%%%%%%%%%%%%%%%%%%%%%%%%%%%%%%%%%%%%%%%%%%%%
\section{Introduction}\label{section.intro}

\subsection{Motivation and results}\label{subsec.motivation}
Let $Q$ be a global function field whose constant field is a finite field $\bF_q$ with $q$-elements of characteristic $p$.
We fix a place  $\infty$ of $Q$ and suppose that its residue field $\bF_\infty$ is equal to $\bF_q$.
Define $A\subset Q$
to be the ring of rational functions regular outside $\infty$.
Let $K$ be a global function field equipped with an injective $\bF_q$-algebra homomorphism $\gamma \colon A\ra K$.
Then $\gamma$ canonically extends to $\gamma\col Q \ra K$, and $K$ is a finite extension of $\gamma(Q)$.
Denote by $G_K:=\Gal(K^\sep/K)$  the absolute Galois group of $K$.
For a place $v$ of $K$, we denote by $G_v$ the decomposition group of $G_K$ at $v$.

Let $\frp \subset A$ be a maximal ideal and set $\bF_{\frp}:=A/\frp$.
 Let $Q_\frp$ be the $\frp$-adic completion of $Q$ and let $A_\frp$ be its valuation ring. 
 In this paper,  we mean by  {\it $\frp$-adic representations} of $G_K$ that finite-dimensional $Q_\frp$-vector spaces $V$ on which $G_K$ acts continuously.
We write $V_v=V\mid_{G_v}$ for the restriction to $G_v$.   
Now
let $V$ and $V'$ be two $\frp$-adic representations of $G_K$ with the same dimension and let $v$ be 
 a finite place (i.e., a place not lying above $\infty$) of $K$. 
We write 
\[
V_v \simeq_\ss V_v'
\]
 if their semi-simplifications are isomorphic as  $\frp$-adic representations of $G_v$.
Since $V_v$ and $V_v'$ have many common properties when $V_v \simeq_\ss V_v'$, it is worth considering when two $\frp$-adic representations are
isomorphic up to semi-simplification.
 To do this, we focus on the case where $V_v$ and $V_v'$ are {\it congruent modulo $\frp$} in the following sense:
For a choice of  $G_v$-stable $A_\frp$-lattices $T \subset V_v$ and $T' \subset V_v'$,  we write 
\[
V_v \equiv_\ss V_v' \pmod \frp
\]
if $T/\frp T$ and $T'/\frp T'$ have isomorphic semi-simplifications $(T/\frp T)^\ss$ and $(T'/\frp T')^\ss$ as  $\bF_\frp$-representations of $G_v$.
As we will see in Remark \ref{rem.sscong}, two $\frp$-adic representations isomorphic up to semi-simplification are congruent modulo $\frp$, and so  
we are interested in the converse:
\begin{question}\label{question}
When does the implication
\[
V_v \equiv_\ss V_v' \pmod \frp \es\es \Lra  \es\es V_v \simeq_\ss V_v'
\]
hold ?
\end{question}

In number field case, Ozeki and  Taguchi \cite{OT14} prove a criterion for two $\ell$-adic Galois representations of the absolute Galois group $G_F$ of a number field $F$ as the following type: {\it Let $v$ be a finite place of $F$. If the prime $\ell$ is sufficiently large $($with respect to $F$, $v$, and the type of the $\ell$-adic representations under considerations$)$, then for any place $u$ lying above $\ell$ and any $\ell$-adic representations $V$ and $V'$ of $G_F$ satisfying a certain set of conditions $($which in particular contains semi-stability at $v$ and $u)$, the two congruence relations $V_v \equiv_\ss V_v' \pmod \ell $ and $V_u \equiv_\ss V_u' \pmod \ell$ imply $V_v \simeq_\ss V_v'$.}
The method of the proof is, as in \cite[Lemma 3.9]{Oze11} and \cite[\S3.3]{RT17}, to recover the Frobenius characteristic polynomials of $V_v$ and $V_v'$ at $v$ from their reduction modulo $\ell$ by computing a  bound of the coefficients of them, which follows from the facts in $p$-adic Hodge theory due to Caruso \cite{Car06} and Caruso-Savitt \cite{CS09} that describe relations between Hodge-Tate weights and {\it tame inertia weights} of torsion semi-stable $p$-adic representations.
Here tame inertia weights are numerical invariants determined by inertia action on the semi-simplifications of residual representations.

Question \ref{question} is meaningful from the point of view of Serre's uniformity problem. 
It is proved by Serre \cite{Ser72} that for any elliptic curve $E$ over $\bQ$ without CM, there exists a constant $C(E)>0$ such that for any prime number $p>C(E)$, the residue representation 
\[
\rho_{E,p}^{}\colon G_\bQ\ra \GL_{\bF_p}(E[p]) \simeq \GL_2(\bF_p)
\] 
is surjective.
In contrast, Serre asked whether $C(E)$ can be made independent of $E$: 
{\it Does there exist $C>0$ such that for any non-CM elliptic curve $E$ over $\bQ$ and any prime number $p>C$, the residue representation $\rho_{E, p}^{}$ is surjective ?}\ 
Replacing the surjectivity condition with another, one can get a variant of Serre's uniformity problem.
For example, as we will see below, a conjecture of Rasmussen and Tamagawa \cite{RT08} is a kind of this problem.
In a function field setting, Pink and R\"utsche \cite{PR09} show a Drinfeld module analog of Serre's result. 
Thus one can consider the uniformity problem for Drinfeld modules.
To solve these problems, we can expect that it is useful to consider the congruences of Galois representations. 
For instance, under the assumption of the Generalized Riemann Hypothesis, 
the Rasmussen-Tamagawa conjecture is proven in \cite{RT17} by showing the implication as in Question \ref{question} for $\ell$-adic  representations coming from abelian varieties.
In this context, we can interpret the work of Ozeki and Taguchi \cite{OT14} as giving a general framework for studying Serre's uniformity problem in number field case, and so it is natural to be interested in function field case.

The purpose of this article is to study the congruence relation between two $\frp$-adic Galois representations arising from {\it effective $A$-motives} over $K$. 
To be more precise, set  $A_K:=A\ot_{\bF_q}K$ and consider the endomorphism 
$\sig$ of $A_K$ given by $\sig(a\ot \lam)=a\ot \lam^q$ for $a \in A$  and $\lam \in K$.
Let  $J=J_K \subset A_K$  be the ideal generated by $\{a\ot 1-1\ot \gamma(a) \mid a \in A \}$, so that it is the kernel of the map $A_K \ra K $ with $ a\ot \lam \mapsto \gamma(a)\lam$. 
An {\it effective $A$-motive over $K$} is  a pair $\ul{M}=(M,\tau_M)$ consisting of a locally free $A_K$-module $M$ of finite rank and an injective  
$A_K$-homomorphism
$\tau_M \colon \sig \ast M:=M\ot_{A_K, \sig}A_K \ra M$ whose cokernel  is a finite-dimensional $K$-vector space annihilated by a power of $J$.
The  {\it rank} and {\it dimension} of $\ul M$  are defined by   $\rk \ul M:=\rk_{A_K}M$ and  
$\dim \ul M:=\dim_K \Cok (\tau_M)$, and $\ul M$ is said to be {\it of height $\leq h$}  if $J^h\cdot \Cok(\tau_{M})=0$.  
The notion of effective $A$-motives is a variant of {\it $t$-motives} in the sense of  Anderson \cite{And86} and it is also a generalization of Drinfeld $A$-modules (cf.\ \cite{Dri74}, \cite{DH87}, and \cite{Hay92})
and abelian Anderson $A$-modules (cf.\ \cite{And86}, \cite{Gos96}, and \cite{Har19}), which are a function field analog of abelian varieties over number fields.
In this article, we consider two kinds of ``cohomology realizations'' of $\ul M$; the first is the {\it $\frp$-adic realizations}; the second is the {\it de Rham realizations}. 
The $\frp$-adic realization $H^1_\frp(\ul M,Q_\frp)$ of $\ul M$ is a $\frp$-adic representation of $G_K$ with dimension equal to $\rk \ul M$.
Taking the $Q_\frp$-linear dual of $H^1_\frp(\ul M, Q_\frp)$, we get the $\frp$-adic rational Tate module $V_\frp\ul M$, which is the central consideration of our study.
The de Rham realization $H^1_\dR(\ul M, K)$ of $\ul M$ is a finite-dimensional $K$-vector space equipped with a descending filtration by $K$-subspaces so-called the {\it Hodge-Pink filtration}.
As jumps of this filtration, we define the multi-set $\HP(\ul M)$ of {\it Hodge-Pink weights} of $\ul M$, which play the role of Hodge-Tate weights.
We should note that we adopt the negatives of usual Hodge-Pink weights as the definition of Hodge-Pink weights; see Remark \ref{rem.HPneg}.

In order to investigate the congruence of  $V_\frp \ul M$, we impose the {\it strong semi-stability} on effective $A$-motives, which is originally defined for {\it analytic $\tau$-sheaves} by Gardeyn; see \cite{Ga02} and \cite{Ga03}.
The Tate module $V_\frp\ul M$ associated with strongly semi-stable $\ul M$  behaves as ``semi-stable representations'' in the following sense.  
If $\ul M$ has strongly semi-stable reduction at a finite place $u$ of $K$ lying above $\frp$, the restriction $V_\frp\ul M|_{G_u}$ is isomorphic (up to semi-simplification) to a direct sum of $\frp$-adic representations arising from {\it effective local shtukas}.
Local shtukas are linear algebraic objects over the valuation ring of $K_u$  describing lattices of {\it equal-characteristic crystalline representations} in Hodge-Pink theory (cf.\ \cite{Har09}, \cite{Ki09}, \cite{GL11}, \cite{Har11}, and \cite{HK20}).
Although there are still no definitions of equal-characteristic analog of Hodge-Tate and semi-stable $p$-adic representations, we can see that $V_\frp\ul M|_{G_u}$ has a behavior expected of equal-characteristic semi-stable representations.
As we will see in Theorem \ref{thm.tamemot}, the Hodge-Pink weights (and the dimension) of $\ul M$ relate with the tame inertia weights of $V_\frp\ul M|_{G_u}$ under some condition on ramification.
This is an analogue of results on semi-stable $p$-adic representations in \cite{Car06} and \cite{CS09}. 
On the other hand, the strong semi-stability of $\ul M$ at a finite place $v\nmid \frp$ implies that $V_\frp\ul M|_{G_v}$ is ``semi-stable'' in the sense that the inertia subgroup $I_v$ of $G_v$ acts on $V_\frp\ul M$ unipotently.
This arrows us to consider the Frobenius characteristic polynomial of $V_\frp\ul M$ at $v$ and define the multi-set $\W_v(\ul M)$ of {\it Weil weights} at $v$ via $\infty$-adic  absolute values of roots of the characteristic polynomial.
Then we will see in Proposition \ref{prop.freigen} that the Weil weights are related to the Hodge-Pink weights of $\ul M$.
In consequence, if $\ul M$ has strongly semi-stable reduction at both $u\mid \frp$ and $v\nmid \frp$, then the coefficients of the Frobenius characteristic polynomial at $v$ can be described by the reduction modulo $\frp$ of $V_\frp\ul M|_{G_u}$ from the perspective of tame inertia weights.

To state our first result, let $\mathsf{Mot}_{K,r,v}(u,h)$ be the set of effective $A$-motives of rank $r$ and of height $\leq h$ which are strongly semi-stable at $v$ and $u$, and satisfy a set of conditions; see Definition \ref{defi.mot}.
Note that we have $v\nmid \frp$ and $u\mid \frp$ under the assumptions of our results (Theorems \ref{thm.m1} and \ref{thm.m2}).
We denote by $d_\frp=[\bF_\frp:\bF_q]$ the degree of $\frp$.
Let $[K:Q]_{\rm i}$ be the inseparable degree of $K/\gamma(Q)$, that is, 
$[K:Q]_{\rm i}=[K:K_{\rm s}]$ for the maximal separable extension $K_{\rm s}$ of $\gamma(Q)$ in $K$.
We define 
\[
D_K=
\begin{cases}
\max\{d_{\frq} \mid \frq\subset A\ \mbox{is a maximal ideal dividing}\  \frd  \} & \mbox{if}\ \frd \neq A,\\
1 & \mbox{if}\ \frd=A,
\end{cases}
\]
where $\frd$ is the relative discriminant of 
$K_{\rm s}/\gamma(Q)$.
Using the strategy as in \cite{OT14}, we obtain the following result: 

\begin{theorem}[{$=$\ Theorem \ref{thm.main1}}]\label{thm.m1}
Let $r$ and $h$ be as above and fix a finite place $v$ of $K$ with degree $d_v=[\bF_v:\bF_q]$.
For any maximal ideal $\frp$ of $A$ with $d_\frp>\max\{d_vr^2h, [K:Q]_{\rm i}h, D_K\}$
and any finite place $u$ of $K$ lying above $\frp$, the following holds: 
For any $\ul M \in \mathsf{Mot}_{K,r,v}(u,h)$ and $\ul M' \in \mathsf{Mot}_{K,r,v}\left(u,(q_\frp-2) [K:Q]_{\rm i}\inv\right)$, if both
\[
\begin{cases}
V_\frp\ul M|_{G_v} \equiv_\ss V_\frp\ul M'|_{G_v} \pmod \frp, & \mbox{and}\\
V_\frp\ul M|_{G_u} \equiv_\ss V_\frp\ul M'|_{G_u} \pmod \frp & 
\end{cases}
\]
hold, 
 then one has 
$V_\frp\ul M|_{G_v} \simeq_\ss V_\frp\ul M'|_{G_v}$, $\dim \ul M=\dim \ul M'$, and $\cW_v(\ul M)=\cW_v(\ul M')$.
\end{theorem}

Our second result (Theorem \ref{thm.main2}) is motivated by a non-existence conjecture on constrained abelian varieties over number fields due to Rasmussen and Tamagawa in \cite[Conjecture 1]{RT08} as follows.
For a fixed number field $F/\bQ$ and a fixed integer $g>0$, the conjecture says that for a prime number $\ell$ large enough (with respect to $F$
and $g$), there exist no $g$-dimensional abelian varieties $\cA$ over $F$ with good reduction outside $\ell$ such that the $G_F$-action on the residue Galois representations $\cA[\ell]$ are given by upper-triangular matrices whose diagonal components are powers of the modulo $\ell$ cyclotomic character. 
The condition on $\cA[\ell]$, in other words, is that the congruence relation $V_\ell\cA \equiv_\ss \bigoplus_{i=1}^{2g}\bQ_\ell(\chi_{\ell}^{n_i}) \pmod \ell$ holds for some integers $n_i$, where $\bQ_\ell(\chi_{\ell}^{n_i})$ is the representation space of the $n_i$-power of the $\ell$-adic cyclotomic character $\chi_\ell \col G_F\ra \bZ_\ell\mal$.
Although the conjecture of Rasmussen and Tamagawa is an open problem in general, various partial results are known; see \cite{Oze11}, \cite{Oze13}, \cite{AM14}, \cite{Bou15}, \cite{RT17}, and \cite{Lom18}.
For example, in the case where considered abelian varieties $\cA$ have everywhere semi-stable reduction, the non-existence of such abelian varieties is proved in \cite[Theorem 3.6]{RT17}. 

Motivated by the above conjecture, we would like to consider the non-existence of effective $A$-motives $\ul M$ whose $\frp$-adic Tate module $V_\frp\ul M$ are ``residually Borel''.
In the case where $Q$ is the rational function field $\bF_q(t)$ and  $A=\bF_q[t]$, 
the author's previous work \cite{Oku19} proves that, under the assumption  $r \geq 2$ and $r \nmid [K:Q]_{\rm i}$,  if $d_\frp$ is sufficiently large, then there exist no rank-$r$ Drinfeld $\bF_q[t]$-modules $\ul E$ over $K$ with good reduction outside $\frp$ such that the $G_K$-action on $T_\frp\ul E\ot_{A_\frp}\bF_\frp$ is given by upper-triangular matrices whose diagonal components are $\bF_\frp\mal$-valued characters arising from tensor powers of the Carlitz motive $\ul \cC$ (Example \ref{ex.carlitz}).
Here the Carlitz motive $\ul \cC$ is a one-dimensional rank-one effective $\bF[t]$-motive with everywhere good reduction, which gives a function field analog of the theory on cyclotomic extensions of number fields.
For general $A$, replacing $\ul \cC$ by one-dimensional rank-one effective $A$-motives with good reduction, we show the following non-existence result:

\begin{theorem}[{$=$\ Theorem \ref{thm.main2}}]\label{thm.m2}
Let $r \geq 2$ be an integer and $h\geq 0$.
Let $v$ be a finite place of $K$.
For any maximal ideal $\frp$  of $A$ with $d_\frp>\max\{d_vr^2h[K:Q]_{\rm i}, D_K\}$ and any finite place $u$ of $K$ lying above $\frp$, there 
exist no effective $A$-motives $\ul M \in \mathsf{Mot}_{K,r,v}(u,h)$ satisfying the following conditions:
\begin{itemize}
\item There exist  at least one Weil weight $w \in \W_v(\ul M)$ such   that $[K:Q]_{\rm i}w$ is a non-integer,
\item There exist one-dimensional rank-one effective $A$-motives $\ul M_1,\ldots,\ul M_r$ over $K$ having good reduction at both
$v$ and $u$ such that 
%\[
%V_\frp\ul M \equiv_\ss \bigoplus_{i=1}^rV_\frp(
%\ul{M}_i^{\ot m_i}) \pmod \frp
%\]
%holds as $G_K$-representations.
\[
\begin{cases}
\DS V_\frp\ul M|_{G_v} \equiv_\ss \bigoplus_{i=1}^rV_\frp(
\ul{M}_i^{\ot m_i})|_{G_v} \pmod \frp
, & \mbox{and}\\
\DS V_\frp\ul M|_{G_u} \equiv_\ss \bigoplus_{i=1}^rV_\frp(
\ul{M}_i^{\ot m_i})|_{G_u} \pmod \frp &
\end{cases}
\]
hold for some non-negative integers $m_1,\ldots,m_r$.
\end{itemize}
\end{theorem}

Remark that we can not expect that the bounds of $d_\frp$ in Theorems \ref{thm.m1} and \ref{thm.m2} are optimal because they are merely a combination of conditions necessary for proof.
However, in terms of the uniformity problem, it is worth the fact that these bounds are independent of effective $A$-motives themselves. 

\subsection{Organization of this article}

The purpose of \S2 is to review the definitions and basic facts on Galois representations used in this article.
In particular, we state that the characteristic polynomials completely determine the semi-simplifications of  Galois representations.
(This is well-known, but few references describe it, so we give the proof  in Appendix.)  
In \S3, we introduce basics about effective $A$-motives, often without proof. 
First, we recall the fundamental concepts on effective $A$-motives over general $A$-fields, and then we study  $\frp$-adic and de Rham cohomology realizations of effective $A$-motives over $K$.
In \S4, following \cite{Ki09} and \cite{HK20}, over a complete discrete valuation field $L$ with a perfect residue field, we study the equal-characteristic Fontaine theory described by local shtukas and torsion local shtukas.
As an analog of Raynaud's classification of finite flat group schemes of $p$-power rank, we classify torsion local shtukas with ``coefficients'' of rank one and investigate  tame inertia weights of residue representations arising from local shtukas; 
 this will be extended to strongly semi-stable effective $A$-motives in the next section.
We introduce in \S5 the notion of {\it analytic $A(1)$-motives}, a rigid analytic version of effective $A$-motives.
Considering rigid analytification, we define and investigate the notion of strongly semi-stable reduction for effective $A$-motives.
In \S6, after preparing properties on  Weil weights, we prove our results.

\subsection{Notation and convention}\label{subsec.notation}

Let $\bF_q, Q, \infty, A,$ and $\gamma\col A\ra K$ be as in the beginning of \S\S \ref{subsec.motivation}.
We use the following notation in this article.

\begin{itemize}
\item We mean by a {\it multi-set}  a collection of unordered numbers allowing for multiplicity.
For a multi-set $\cX$, we define 
\[
\Sigma\cX:=\sum_{x\in \cX}x
\] if the right hand side is well-defined.
For example, if $\cX=\{x,x,x\}$, then $\Sigma\cX=3x$.
\item For a ring homomorphism $\sigma\col R \ra R$ and an $R$-module $M$, we set $\sig\ast M:=M\ot_{R,\sig}R$ and $\sig\ast f:=f\ot\id\col \sig\ast M\ra \sig\ast N$ for an $R$-homomorphism $f\col M\ra N$.
For any $m\in M$, we write $\sig\ast m:=m\ot 1 \in \sig\ast M$.
\item Let $v$ be a place of $K$. Then the symbols $K_v$, $\cO_v,$ and $\bF_v$  mean the completion of $K$ at $v$, the valuation ring of $K_v$, and the residue field of $K_v$, respectively.
Denote by $d_v:=[\bF_v:\bF_q]$ the degree of $v$ over $\bF_q$.
For each $v$, we fix an embedding $\bar K \ra \bar K_v$ between fixed algebraic closures of $K$ and $K_v$, and identify the decomposition group $G_v \subset G_K$ as $G_{K_v}=\Gal(K_v^\sep/K_v)$.
\item For the fixed place $\infty$ of $Q$ whose residue field is $\bF_q$, let $|\cdot|_\infty\col Q\ra \bR$ be the corresponding absolute value normalized such as 
$|a|_\infty=\#A/aA$ for any $a\in A$. 
We fix an algebraic closure $\bar Q$ of $Q$ and extend the absolute value to $\bar Q$, which is again denoted by $|\cdot|_\infty$.
\end{itemize}  

For the absolute value $|\cdot|_\infty$, we have the following elementally lemma used to compare Frobenius characteristic polynomials to prove the main results:

\begin{lemma}\label{lem.absvalue}
Let $\frp$ be a maximal ideal of $A$.
If $a\in A$ satisfies $a\equiv 0 \pmod \frp$ and $|a|_\infty < q^{d_\frp}$, then $a=0$. 
\end{lemma}

\begin{proof}
Assume that $a\neq 0$.
Then by  $a\equiv 0 \pmod \frp$, the principal ideal $aA$ decomposes as  
$aA=\frp^e\fra$, where $e>0$ and $\frp \nmid \fra$.
Thus it follows that 
\[
|a|_\infty=\#A/aA=(\#A/\frp^e)\cdot (\# A/\fra) \geq q^{d_\frp}.
\]
This is a contradiction.
\end{proof}

%%%%%%%%%%%%%%%%%%%%%%%%%%%%%%%%%%%%%%%%%%%%%%%%%%%%%%%%%%%%%%%%%%%%%%%%%%%%%%%%%%%
%
% 2. Preliminary
%
%%%%%%%%%%%%%%%%%%%%%%%%%%%%%%%%%%%%%%%%%%%%%%%%%%%%%%%%%%%%%%%%%%%%%%%%%%%%%%%%%%%%
\section{Preliminary}
This section recalls some basic facts on Galois representations used in this article. 
Let $G$ be a profinite group and 
%the absolute Galois group of either $K$ or a completion $K_v$ at a place $v$ of $K$. 
$F$  a topological field.
(We consider the discrete topology on $F$ if it is finite.) 
On finite-dimensional $F$-vector spaces, we consider the product topology induced by $F$.
By an {\it $F$-representation of $G$}, we mean a finite-dimensional $F$-vector space with a continuous $G$-action.
For an $F$-representation $V$ of  $G$, the Jordan-H\"{o}lder theorem gives a descending filtration by  $G$-stable $F$-subspaces
\[
\{0\}=V_0 \subset V_1 \subset \cdots \subset V_{n}=V
\] 
such that each $V_{i}/V_{i-1}$ is irreducible. 
Then the {\it semi-simplification} of $V$ is defined by 
\[
V^\ss:=\bigoplus_{i=1}^n V_i/V_{i-1},
\]
which is completely determined by $V$ up to isomorphisms.

\begin{definition}
For two $F$-representations $V$ and $V'$ of $G$, we write $V \simeq_\ss V'$ if their semi-simplifications $V^\ss$ and ${V'}^{, \ss}$ are isomorphic as $F$-representations of $G$.
\end{definition}

For any element $s \in G$, we denote the characteristic polynomial of $s$ by $P_{V,s}(X):=\det(X-s\mid V)$.
Since the characteristic polynomials are multiplicative in short exact sequences, 
two $F$-representations $V$ and $V'$ with $V\simeq_\ss V'$ satisfy $P_{V,s}(X)=P_{V',s}(X)$ for any $s \in G$.
It is well-known that the converse also holds.
This fact is crucial for our study.

\begin{theorem}[Brauer-Nesbitt]\label{thm.BN}
Two $F$-representations $V$ and $V'$ of $G$ satisfy $V\simeq_\ss V'$ if and only if $P_{V,s}(X)=P_{V',s}(X)$ for any $s\in G$.  
\end{theorem}

\begin{proof}
The proof is given in Appendix. 
Note that it is often called the Brauer-Nesbitt theorem in the case where $G$ is a finite group; see \cite[Theorem 30.16]{CR62}. 
\end{proof}

Now for  a maximal ideal
$\frp$ of $A$ and a place $v$ of $K$, let us consider the case where $F=Q_\frp$ and $G=G_v=\Gal(K_v^\sep/K_v)$.
We denote by $I_v$ the inertia subgroup of $G_v$.
For a $\frp$-adic representation $V$ of $G_v$, we say that $V$ is {\it unramified $($at $v$$)$} if $I_v$ trivially acts on $V$.
Let $\phi_v$ be the arithmetic Frobenius of $G_{\bF_v}=\Gal(\bF_v^\sep/\bF_v)$, that is, $\phi_v(x)=x^{q_v}$ for $x\in \bF_v^\sep$.
Take a lift  $\Frob_v \in G_v$ of $\phi_v$ via $G_v \thra G_{\bF_v}$ and set $P_{V,v}(X):=P_{V,\Frob_v}(X)$.
Then the Brauer-Nesbitt theorem particularly implies the following:

\begin{proposition}\label{prop.chp}
Let $V$ and $V'$ be $\frp$-adic representations of $G_v$ with the same dimension. 
Suppose that $I_v$ acts unipotently on both $V$ and $V'$.
Then $V \simeq_\ss V'$ if and only if $P_{V, v}(X)=P_{V', v}(X)$.
\end{proposition} 

\begin{proof}
The ``only if'' part is trivial, so we prove the converse.
Suppose that $P_{V, v}(X)=P_{V', v}(X)$.
Considering semi-simplification, we may assume that $V$ and $V'$ are semi-simple so that they are unramified at $v$ by assumption.
Thus $V$ and $V'$ are $F$-representations of $G_{\bF_v} (\cong G_v/I_v)$, and so we have 
$P_{V,\phi_v}(X)=P_{V',\phi_v}(X)$.
Since the pro-cyclic group $G_{\bF_v}$ is topologically generated by $\phi_v$, it follows by continuity  that $P_{V,s}(X)=P_{V',s}(X)$ for any $s\in G_{\bF_v}$. 
Thus we obtain the conclusion by the Brauer-Nesbitt theorem.
\end{proof}

\begin{definition}
Let $V$ and $V'$ be $\frp$-adic representations of $G_v$ with the same dimension.
We say that $V$ and $V'$ are {\it congruent modulo} $\frp$  and write  $V \equiv_\ss V' \pmod \frp$ if for a choice of $G_v$-stable $A_\frp$-lattices $T \subset V$ and $T' \subset V'$, we have $T/\frp T \simeq_\ss T'/\frp T'$ as $\bF_\frp$-representations of $G_v$.
\end{definition}

\begin{remark}\label{rem.sscong}
Since the $G_v$-action is continuous, there exists a $G_v$-stable $A_\frp$-lattice $T$ of $V$ and  hence $P_{V,s}(X)=\det(X-s \mid T)\in A_\frp[X]$ for any $s\in G_v$.  
For any choice of $T \subset V$, the characteristic polynomial of $s\in G_v$ on $T/\frp T$ equals to $P_{V,s}(X) \pmod \frp \in \bF_\frp[X]$. 
By the Brauer-Nesbitt theorem, the semi-simplification $(T/\frp T)^\ss$ is independent of the choice of lattices and hence that $V\simeq_\ss V'$ implies $V\equiv_\ss V'\pmod \frp$.
\end{remark}

%%%%%%%%%%%%%%%%%%%%%%%%%%%%%%%%%%%%%%%%%%%%%%%%%%
%%%%%%%%%%%%%%%%%%%%%%%%%%%%%%%%
%
%  3. Effective A-motive
%
%%%%%%%%%%%%%%%%%%%%%%%%%%%%%%%%%%%%%%%%%%%%%%%%%%%%%%%%%%%%%%%%%%%%%%%%%%%%%%%%%%%%
\section{Effective $A$-motives}

%%%%%%%%%%%%%%%%%%%%%%%%%%%%%%%%%%%%%%%%%%%%%%%%%%%%%%%%%%%%%%%
\subsection{Definitions} 
In this first subsection, 
we recall the notion of effective $A$-motives over general fields.
Let $(F,\gamma)$ be an {\it $A$-field}, that is, an extension  field of $\bF_q$ equipped with a (not necessary injective) $\bF_q$-algebra homomorphism 
$\gamma \colon A \ra F$. 
We set $A_F:=A\ot_{\bF_q}F$ and consider the ideal $J:=J_F:=\Ker(A_F \ra F; a\ot \lam\mapsto \gamma(a)\lam) \subset A_F$, which is generated by $\{a\ot1-1\ot \gamma(a) \mid a\in A\}$.
Note that $A_F$ is a Dedekind domain.
For the endomorphism $\sig=\sig_F:=\id_A\ot (\cdot)^q$ of $A_F$ and an $A_F$-module $M$, we set $\sig\ast M:=M\ot_{A_F,\sig}A_F$
 and write $\sig\ast m=m\ot 1 \in \sig\ast M$ for $m \in M$.

\begin{definition}
An {\it effective $A$-motive over an $A$-field $F$ of rank} $r$ is a pair $\ul M=(M,\tau_M)$ consisting of a locally free $A_F$-module of rank $r$, and an injective $A_F$-homomorphism $\tau_M\colon \sig\ast M \ra M$ such that $\Cok(\tau_M)$ is a finite-dimensional $F$-vector space  annihilated by some power of $J$.
We set $\rk\ul M:=r$ and define the {\it dimension} of $\ul M$ by $\dim\ul M:=\dim_F\Cok(\tau_{M})$.
For $h \geq 0$,
we say that $\ul M$ is  {\it of height $\leq h$}  if  $J^h\cdot \Cok(\tau_M)=0$.

A {\it morphism} $f\colon \ul M \ra \ul N$ between effective $A$-motives over $F$ is an $A_F$-homomorphism $f\colon M \ra N$ such that 
$f\circ \tau_{M} = \tau_{N} \circ \sig\ast f$.
A morphism $f\colon \ul M \ra \ul N$ is called an {\it isogeny} if $f$ is injective with torsion cokernel. 
%We denote by $\Hom_F(\ul M,\ul N)$ the $A$-module of morphisms between $\ul M$ and $\ul N$.
%We also set $\End_F(\ul M):=\Hom_F(\ul M,\ul M)$.
\end{definition}

\begin{remark}
For a locally free $A_F$-module $M$ and an $A_F$-homomorphism $\tau_M\colon \sig\ast M \ra M$, if $\Cok(\tau_M)$ is annihilated by a power of $J$, then $\tau_M$ is injective and $\Cok(\tau_M)$ is finite-dimensional over $F$ by \cite[Proposition 2.3]{Har19}.
\end{remark}

\begin{example}[Unit objects]
Let $M=A_F$ and define $\tau_{M}\col \sig\ast M\ra M$ by $\tau_{M}:=\sig\ot \id$.
Then under the natural isomorphism $\sig\ast M=A_F\ot_{A_F,\sig}A_F \cong A_F;  a\ot b \mapsto \sig(a)b$, the map $\tau_{M}$ is identified as $\id\col A_F \ra A_F$. 
We denote by $\ul {A_F}=(A_F,\id)$, which is an effective $A$-motive over $F$ with $\rk \ul{A_F}=1$ and $\dim \ul{A_F}=0$, so that it is of height $\leq 0$.
\end{example}

\begin{example}[Abelian Anderson $A$-modules]
Let $r$ and $d$ be positive integers.
An {\it abelian Anderson $A$-module over $F$ of rank $r$ and dimension $d$} is a pair $\ul E=(E, \phi)$ consisting of an affine group scheme $E$ satisfying $E\cong \bG_{a,F}^d$ as $\bF_q$-module schemes over $F$ and a ring homomorphism $\phi\col A\ra \End_{F}(E); a\mapsto \phi_a$ such that  
\begin{itemize}
\item[(1)] $(\Lie \phi_a-a)^d=0$ on $\Lie E$ for any $a\in A$,
\item[(2)] the set $M:=M(\ul E):=\Hom_{\SSC{F\mbox{\tiny-groups},\ \bF_q\mbox{\tiny -lin.}}}(E,\bG_{a,F})$ of $\bF_q$-liner homomorphisms of $F$-group schemes is a locally free $A_F$-module of rank $r$ under the action given on $m\in M$ by $(a\ot \lam)m:= \lam\circ m\circ \phi_a$ for $a\in A$ and $\lam \in F$.
\end{itemize}
If $d=1$, then $\ul E$ is called a {\it Drinfeld $A$-module}.
Note that in the case where $A=\bF_q[t]$ and $F$ is perfect, an abelian Anderson $A$-module coincides with an abelian $t$-module in the sense of Anderson \cite{And86}. 
Let $\tau$ be the relative $q$-Frobenius endomorphism of $\bG_{a,F}=\Spec F[x]$ given by $x\mapsto x^q$.
For  an abelian Anderson $A$-module $\ul E$ of rank $r$ and dimension $d$ over $F$ and $M=M(\ul E)$,  define $\tau_{M}\col\sig\ast M\ra M$ by $\tau_M(\sig\ast m)=\tau\circ m$ for $m\in M$.
Then it is known that the pair $\ul{M}(\ul E)=(M(\ul E),\tau_{M})$ becomes an effective $A$-motive over $F$ of rank $r$ and dimension $d$, and that the  correspondence $\ul E \mapsto \ul{M}(\ul E)$ determines a contravariant fully faithful functor; see \cite[Theorem 1]{And86} and \cite[Theorem 3.5]{Har19}.
\end{example}

\begin{example}[the Carlitz motive]\label{ex.carlitz}
Let $A=\bF_q[t]$ and $F=\bF_q(\vartheta)$ the rational function field in variable $\vartheta$.
Suppose that $\gamma\col A\ra F$ is given by $\gamma(t)=\vartheta$, so that $A_F=\bF_q(\vartheta)[t]$ and $J=(t-\vartheta)$.
Then the {\it Carlitz motive} is given by $\ul \cC=(\bF_q(\vartheta)[t], \tau_{\cC}=t-\vartheta)$.
It is known that $\ul\cC$ is an effective $\bF_q[t]$-motive associated with the {\it Carlitz module}, which is a Drinfeld $\bF_q[t]$-module of rank one.
\end{example}

\begin{definition}
Let $\ul M$ and $\ul M'$
be effective $A$-motives over $F$.
Then the {\it direct sum} of $\ul M$ and $\ul M'$ is defined by $\ul M\oplus \ul M':=(M\oplus M',\tau_{M}\oplus\tau_{M'})$, and 
the {\it tensor product} of $\ul M$ and $\ul M'$ is defined by
$\ul M\ot \ul M':=(M\ot_{A_F} M', \tau_M\ot \tau_{M'})$.
Note that both $\ul M\oplus \ul M'$ and $\ul M\ot\ul M'$ are effective $A$-motives over $F$.
We also define the {\it tensor powers} of $\ul M$ by $\ul M^{\ot 0}:=\ul{A_F}$ and by $\ul M^{\ot n}:=\ul M^{\ot (n-1)}\ot \ul M$ for $n>0$. 
We define the {\it determinant} of $\ul M$ by $\det \ul M:=\wedge^{\rk\ul M}\ul M$.
\end{definition}

The next formulas follow easily from the elementary divisor theorem (cf.\ \cite[Page 51]{HJ20}):

\begin{proposition}\label{prop.rkdim}
Let $\ul M$ and $\ul M'$ be effective $A$-motives over $F$.
\begin{itemize}
\item[$(1)$] $\rk \ul M\oplus \ul M'=(\rk M)+(\rk \ul M')$ and $\dim M\oplus \ul M'=(\dim M)+(\dim \ul M')$.
\item[$(2)$] $\rk \ul M\ot \ul M'=(\rk \ul M)\cdot (\rk \ul M')$ and 
$\dim \ul M\ot \ul M'=(\rk \ul M')\cdot(\dim \ul M)+(\rk \ul M)\cdot (\dim \ul M')$.
\item[$(3)$] $\rk \det \ul M=1$ and $\dim \det\ul M=\dim \ul M$.
\end{itemize}
\end{proposition}

Let $F\ra F'$ be a field embedding over $\bF_q$  and regard $F'$ as an $A$-field via $A\overset{\gamma}{\lra} F \ra F'$.
Then the induced ring homomorphism $A_F \ra A_{F'}$ commutes the diagram
\[
\xymatrix{
A_F \ar[r]\ar[d]_{\sig_F} & A_{F'}\ar[d]^{\sig_{F'}} \\
A_F \ar[r] & A_{F'}.
}
\]
Thus for an effective $A$-motive $\ul M=(M,\tau_M)$   over $F$,  we have 
\[
\sig\ast(M\ot_{A_F}A_{F'})\cong \sig\ast M \ot_{A_F} A_{F'} \overset{\tau_M\ot \id}{\lra} M\ot_{A_F}A_{F'},
\]
which provides the base change $\ul M\ot A_{F'}$ of $\ul M$ as follows:

\begin{proposition}
If $F$ and $F'$ be as above, then $\ul M=(M,\tau_M) \mapsto \ul M\ot A_{F'}:=(M\ot_{A_F}A_{F'}, \tau_M\ot\id)$ defines a functor from the category
of effective $A$-motives over $F$ of height $\leq h$ to that of effective $A$-motives over $F'$ of height $\leq h$ such that 
$\rk \ul{M}=\rk \ul{M}\ot A_{F'}$ and  $\dim\ul{M}=\dim\ul{M}\ot A_{F'}$.
\end{proposition}

%%%%%%%%%%%%%%%%%%%%%%%%%%%%%%%%%%%%%%%%%%%%%%%%%%%%%%%%%%%%%%%
\subsection{$\frp$-adic realizations}

In what follows, let us consider effective $A$-motives defined over the $A$-field $(K, \gamma \colon A \ra K)$ as in $\S 1$ (in particular, $\gamma$ is injective).
Let $\frp$ be a maximal ideal of $A$.
For an effective $A$-motive $\ul M$ over $K$, we obtain $\frp$-adic  representations of 
$G_K$ as follows.
Let $A_{\frp,K^\sep}:=A_\frp \wh\ot_{\bF_q}K^\sep =\ilim A_{K^\sep}/\frp^n A_{K^\sep}$ be the $\frp$-adic completion of 
$A_{K^\sep}$.
Then the endomorphism $\sig$ of $A_{K^\sep}$ canonically extends to $\sig \col A_{\frp,K^\sep}\ra A_{\frp,K^\sep}$  
and so $\tau_M$ induces $\tau_M \col \sig \ast M \ot_{A_K} A_{\frp,K^\sep} \ra M \ot_{A_K} A_{\frp,K^\sep}$.
Note that  the $\sig$-invariant subring of $A_{\frp,K^\sep}$ is $(A_{\frp,K^\sep})^\sig=A_\frp$.
Then the {\it $\frp$-adic realization} of $\ul M$ is defined to be 
the $A_\frp$-module
\[
H_\frp^1(\ul M,A_\frp):=(M\ot_{A_K}A_{\frp,K^\sep})^\tau
:=\{ m \in M\ot_{A_K}A_{\frp,K^\sep}\mid \tau_M(\sig\ast m)=m \},
\]
where $\sig \ast m:=m\ot 1$ is the image by the natural map
$M\ot_{A_K}A_{\frp,K^\sep} \ra \sig\ast M\ot_{A_K}A_{\frp,K^\sep}$.
It is known that  $H_\frp^1(\ul M,A_\frp)$ is a free $A_\frp$-module of rank equal to $\rk \ul M$ with a continuous action of $G_K$; see \cite{TW96}.
Define the {\it rational $\frp$-adic realization} of $\ul M$ by 
\[
H_\frp^1(\ul M, Q_\frp):=H_\frp^1(\ul M, A_\frp)\ot_{A_\frp}Q_\frp.
\]
Taking dual of them, we define the {\it $\frp$-adic Tate module} and the {\it rational 
$\frp$-adic Tate module}  
of $\ul M$ by 
\[
T_\frp\ul M:=\Hom_{A_\frp}(H_\frp^1({\ul M}, A_\frp),A_\frp) \es \mbox{and}\es  V_\frp\ul M:=T_\frp\ul M \ot_{A_\frp} Q_\frp.
\]
It follows that  $\ul M \mapsto V_\frp\ul M$ defines a contravariant exact faithful tensor functor from the category of effective $A$-motives over $K$ to that of $\frp$-adic representations of $G_K$. 

\begin{remark}
Set $M_{\frp,K^\sep}=M\ot_{A_K}A_{\frp,K^\sep}$. 
Then $T_\frp\ul M$ can be defined directly by 
\[
T_\frp\ul M=\{f \in \Hom_{A_{\frp,K}}(M_{\frp,K^\sep}, A_{\frp,K^\sep}) \mid f(\tau_M(\sig\ast m))=\sig(f(m))\ \mbox{for}\ m \in M_{\frp,K^\sep}\}.
\]
If $\ul M$ comes from an abelian Anderson $A$-module $\ul E$ over $K$, then $T_\frp\ul M$ is canonically isomorphic to the $\frp$-adic Tate module $T_\frp\ul E=\ilim \ul E[\frp^n]$ of $\ul E$.
\end{remark}

\begin{remark}
For a finite place $v$ of $K$, the (rational) $\frp$-adic Tate modules of effective $A$-motives over $K_v$ are defined in the same way.
By construction, for an effective $A$-motive $\ul M$ over $K$, we have $T_\frp\ul M|_{G_v} \cong T_\frp(\ul M\ot A_{K_v})$ and $V_\frp\ul M|_{G_v}\cong V_\frp(\ul M\ot A_{K_v})$.
\end{remark}
To review the notion of good reduction for effective $A$-motives, let $v$ be a finite place of $K$ and regard $K_v$ as an $A$-field via $\gamma\colon A \ra K \subset K_v$.
Since this map factors through the valuation ring $\cO_v$, the residue field $\bF_v$ also becomes an $A$-field via
$A \overset{\gamma}{\lra} \cO_v \thra \bF_v$.
We set $A_{\cO_v}=A\ot_{\bF_q}\cO_v$.

\begin{definition}
Let $\ul M$ be an effective $A$-motive over $K_v$.

\begin{itemize}
\item[(1)] A {\it model} of $\ul{M}$ is a pair $\ul{\cM}=(\cM,\tau_{\cM})$ consisting of a finite locally free $A_{\cO_v}$-module and an injective $A_{\cO_v}$-homomorphism $\tau_{\cM}\col \sig\ast \cM \ra \cM$ such that there is an isomorphism $\iota \col  M \rai  \cM\ot_{A_{\cO_v}} A_{K_v}$ with $\tau_{\cM}\circ \sig\ast\iota=\iota \circ \tau_{M}$.
\item[(2)] A model $\ul{\cM}$ of $\ul{M}$ is called a {\it good model} if the induced $A_{\bF_v}$-homomorphism
\[
\tau_{\cM}\ot \id \colon \sig\ast\cM\ot_{A_{\cO_v}}A_{\bF_v}\lra  \cM\ot_{A_{\cO_v}}A_{\bF_v}
\]
is injective.
In this case, we say that $\ul{M}$ has {\it good reduction}.
\end{itemize}
We say that an effective $A$-motive $\ul M$ over $K$ has {\it good reduction at $v$} if so does $\ul M\ot K_v$.
\end{definition}

\begin{remark}
For a model $(\cM,\tau_{\cM})$ of an effective $A$-motive $\ul M$ over $K_v$, \cite[Theorem 4.7]{HH16} implies that it is a good model if and only if $\Cok(\tau_{\cM})$ is a finite free $\cO_v$-module annihilated by a power of $J_{\cO_v}$, where  $J_{\cO_v}$ is the ideal of $A_{\cO_v}$ generated by $\{a\ot 1-1\ot \gamma(a) \mid a\in A\}$. 
Therefore, in this case, the pair $(\cM\ot_{A_{\cO_v}}A_{\bF_v},\tau_{\cM}\ot\id)$ becomes an effective $A$-motive over $\bF_v$ which has the same rank and dimension of $\ul M$. 
Note that by \cite[Lemma 2.10]{Ga03} an effective $A$-motive $\ul M$ over $K$ has good reduction at almost all finite places $v$. 
\end{remark}

As an analog of the N\'eron-Ogg-Shafarevich criterion for abelian varieties, Gardeyn \cite{Ga02} proves the following theorem:

\begin{theorem}[Gardeyn]
Let $v$ be a finite place of $K$.
For an effective $A$-motive $\ul{M}$ over $K$, the following statements are equivalent.
\begin{itemize}
\item[$(1)$] $\ul{M}$ has good reduction at $v$.
\item[$(2)$] For all maximal ideal $\frp$ of $A$ with $v\nmid \frp$, the  $\frp$-adic representation $V_\frp \ul{M}$ is unramified at $v$.
\item[$(3)$] There is a maximal ideal $\frp$
with $v\nmid \frp$ such that $V_\frp\ul{M}$ is unramified at $v$. 
\end{itemize}
\end{theorem}

For 
 a finite place $v$ of $K$ with $v\nmid \frp$ and 
 an effective $A$-motive $\ul M$ with good reduction at $v$, 
 the  characteristic polynomial
\[
P_{\ul M,v}(X)=\det(X-\Frob_v \mid V_\frp\ul M)
\] 
of $\Frob_v$ is well-defined.

\begin{proposition}
If $\ul M$ has good reduction at $v\nmid \frp$, then $P_{\ul M,v}(X)$ has coefficients in $A$
which are independent of $\frp$.
\end{proposition}

\begin{proof}
This is a special case of \cite[Theorem 7.3]{Ga02}.
\end{proof}

We will recall the notion of {\it purity} 
%and {\it mixedness} 
for $A$-motives.
Let $\infty_K$ be  a place of $K$ lying above $\infty$, and denote by  $\bC:=\wh{\ol{K}}_{\infty_K}$ 
 a completion of an algebraic closure of a completion of $K$ at $\infty_K$.
We fix an embedding $K \ra \bC$ 
and  regard $\bC$ as an $A$-field via $\gamma\colon A \overset{\gamma}\lra K  \ra \bC$.
Choose a uniformizer $z_\infty \in Q$ at $\infty$.
Since we now assume $\bF_\infty=\bF_q$, we  have $Q_\infty=\bF_q\izp$ and $A \subset Q_\infty$, 
so that we have  $Q_\infty\wh{\ot}_{\bF_q}\bC=\bC\izp$ and an injection $A_{\bC} \hra \bC\izp$.
The endomorphism $\sig$ of $A_{\bC}$ extends to $\sig\colon \bC\izp \ra \bC\izp$ by $\sig(z_\infty)=z_\infty$
and $\sig(\lam)=\lam^q$ for $\lam\in \bC$.

\begin{definition}
An effective $A$-motive $\ul M$ over $\bC$ is said to be {\it pure} if 
there is a $\bC\izb$-lattice $M_\infty$ of $M\ot_{A_{\bC}}\bC\izp$ such that  for some integers $d,r$ with $r>0$, the 
map $\tau_M^r:=\tau_M\circ \sig\ast \tau_M\circ \cdots \circ \sig^{(r-1)*}\tau_M\colon \sig^{r*}M\ra M$ induces an isomorphism 
\[
z^d_\infty\tau_M^r\colon \sig^{r*}M_\infty \overset{\sim}{\lra} M_\infty.
\]
In this case, the {\it weight} of $\ul M$ is defined by $\Wt {\ul M}:=\frac{d}{r}$.
%\item[(b)] An effective $A$-motive $\ul M$ is said to be {\it mixed}  if there is an increasing filtration by saturated $A$-submotives $\ul M_\mu$ for $\mu \in \bQ$ (i.e., the underlying $M_\mu$ are  saturated $A_{\bC}$-submodules of $M$) such that 
%all graded pieces $\Gr_\mu\ul M:= \ul M_\mu/\bigcup_{\mu'<\mu}\ul M_{\mu'}$ are effective pure $A$-motives of wights $\mu$ and $\sum_{\mu\in \bQ}\rk \Gr_\mu\ul M=\rk \ul M$.
An effective $A$-motive $\ul M$ over $K$ is said to be {\it pure} 
%(resp.\ {\it mixed}) 
if so is $\ul M\ot{A_\bC}$.
% (resp.\ mixed).
\end{definition}

\begin{example}
Let $\ul E$ be a Drinfeld $A$-module over $K$ of rank $r$.
Then the associated effective $A$-motive $\ul M(\ul E)$ is pure of weight $\frac{1}{r}$.
\end{example}

\begin{proposition}\label{prop.pureweight}
Let $\ul M$ and $\ul M'$ be effective pure $A$-motives over $K$.
\begin{itemize}
\item[$(1)$] The weight of $\ul M$ is $\Wt \ul M=(\dim \ul M)/(\rk \ul M)$. 
\item[$(2)$] The tensor product $\ul M \ot \ul M'$ is  pure of weight $(\Wt \ul M)+(\Wt \ul M')$.
\item[$(3)$] If $\ul M$ has good reduction at  a finite place $v$ of $K$, then any root $\alpha \in \bar Q$ of $P_{\ul M,v}(X)$ satisfies
$|\alpha|_\infty=q_v^{\Wt \ul M}$.
\end{itemize} 
\end{proposition}

\begin{proof}
 (1) and (2) are well-known; see \cite[Proposition 2.3.11]{HJ20} for example.
 For a maximal ideal $\frp$ of $A$ with $v\nmid \frp$,
each eigenvalue of $\Frob_v\inv$ on $H_\frp^1(\ul M,A_\frp)$ has the absolute value $q_v^{\Wt \ul M}$
by \cite[Proposition 2.3.36]{HJ20}, which implies  (3).
% since $P_{\ul M,v}(X)=\det(X-\Frob_v\inv\mid H_\frp^1(\ul M,A_\frp))$.
\end{proof}

%\begin{defi}
%An effective $A$-motive $\ul{M}$ over $\bC$ of rank one with $\tau_{M}(\sig\ast M)=J\cdot M$ is called a {\it Carlitz-Hayes $A$-motive}.  
%It is of dimension one and pure of weight one; see \cite[Example 2.3.9]{HJ20}
%\end{defi}
%
%It is known that Carlitz-Hayes $A$-motives always exist by \cite[Example 2.3.6]{HJ20} for example.
%Moreover, under our assumption that the residue field of the place $\infty$ is $\bF_q$, the following assertion holds:

\begin{lemma}\label{lem.purerkone}
Under the assumption $\bF_\infty=\bF_q$,
any rank-one effective $A$-motive over $K$ is pure.
In particular, the determinant $\det \ul M$ of an effective $A$-motive $\ul M$ is pure of weight equal to $\dim\ul M$.
\end{lemma} 

\begin{proof}
Suppose that $\ul M$ is of rank one and set $d=\dim \ul M$.
Then by \cite[Examples 2.3.6 and 2.3.9]{HJ20}, there is an isogeny $f\colon \ul M\ot A_\bC \ra \ul N^{\ot d}$, where $\ul N=(N,\tau_N)$ is an effective pure $A$-motive over $\bC$ of rank one satisfying $\tau_{N}(\sig\ast N)=J\cdot N$, so that $\dim\ul N=1$.
(Such $\ul N$ is so-called a {\it Carlitz-Hayes $A$-motives}.)
Since $\ul N^{\ot d}$ is pure, it follows by \cite[Proposition 2.3.11.\ (d)]{HJ20} that 
$\ul M$ is also pure.
\end{proof}
%%%%%%%%%%%%%%%%%%%%%%%%%%%%%%%%%%%%%%%%%%%%%%%%%%%%%%%%%%%%%%%
\subsection{de Rham realizations and Hodge-Pink weights}
Let $z \in Q$ be  a uniformizer at some place of $Q$ and set $\zeta:=\gamma(z) \in K$.
Then for the ideal $J=\Ker(A_K\thra K)$,  
\cite[Lemma 2.1.3]{HJ20} arrows us to  identify $\ilim A_K/J^n=K\zzb$, where $K\zzb$ is  the power series ring  over $K$ in ``variable'' $\zz$.
Thus we obtain an injective flat homomorphism $A_K \ra K\zzb$ and $J\cdot K\zzb=(\zz)$.

Let $\ul M=(M,\tau_M)$ be an effective $A$-motive over $K$.
The {\it de Rham realizations} of $\ul M$ are defined as
\begin{align*}
H^1_\dR(\ul M, K\zzb)&:=\sig\ast M \ot_{A_K}K\zzb, \\
H^1_\dR(\ul M,K\zzp)&:=H^1_\dR(\ul M, K\zzb)\ot_{K\zzb}K\zzp, \\
H^1_\dR(\ul M,K)&:=\sig\ast M \ot_{A_K}A_K/J \\
&=H^1_\dR(\ul M, K\zzb) \ot_{K\zzb}K\zzb/(\zz) .
\end{align*}
Since $\Cok(\tau_M)$ is annihilated by a power of $J$, the map $\tau_M\colon \sig\ast M\ra M$ extends to an isomorphism 
$
\tau_M\colon H^1_\dR(\ul M,K\zzp) \overset{\sim}{\lra} M\ot_{A_K}K\zzp.
$
The {\it Hodge-Pink lattice} of $\ul M$ is defined as the $K\zzb$-submodule
$
\cQ_{\ul M}:=\tau_{M}\inv(M\ot_{A_K}K\zzb) 
$ 
of $H_\dR^1(\ul M,K\zzp)$.
We set $\cP_{\ul M}=H^1_\dR(\ul M, K\zzb)$, so that  $\cP_{\ul M} \subset \cQ_{\ul M}$.
Then the
 {\it Hodge-Pink filtration} $F^{\bullet}H^1_\dR(\ul M,K)=\{F^i H^1_\dR(\ul M,K)\}_{i \in \bZ}$ of $\ul M$ is the descending filtration by $K$-subspaces of $H^1_\dR(\ul M,K)$ defined by 
\[
F^i H^1_\dR(\ul M,K):=\left(\cP_{\ul M}\cap (\zz)^i\cQ_{\ul M}\right)/\left((\zz)\cP_{\ul M}\cap (\zz)^i\cQ_{\ul M}\right).
\] 
In other words,  the filtration is defined  by letting  $F^i H^1_\dR(\ul M,K)$ to be the image of $\left(\cP_{\ul M}\cap (\zz)^i\cQ_{\ul M}\right)$ via  the surjection 
$
\cP_{\ul M} \thra \cP_{\ul M}/(\zz)\cP_{\ul M}=H^1_\dR(\ul M, K)$.
Thus  we  obtain the following diagram
\[
\xymatrix@C=76pt@R=-5pt@M=6pt{
H^1_\dR(\ul M,K\zzp) \ar[r]^-{\sim}_-{\tau_{M}}& M\ot_{A_K}K\zzp \\
\cup & \\
\cQ_{\ul M} \ar[rdd]^-{\sim}&  \\
\cup & \\
\cP_{\ul M}=H^1_\dR(\ul M, K\zzb)\ar@{^{(}-_>}[r]_-{\tau_M} & M\ot_{A_K}K\zzb . \ar@{^{(}-_>}[uuuu]
}
\]
Since $K\zzb$ is a principal ideal domain, the elementary divisor theorem implies that
there are non-negative integers $h_1\leq \cdots \leq h_{\rk \ul M}$ and $\lam_1,\ldots,\lam_{\rk \ul M} \in K\zzb\mal$ such that $\tau_M \colon \cP_{\ul M} \ra M\ot_{A_K}K\zzb$ is of the form
\begin{equation}\label{eq.tau}
\tau_M=
\begin{pmatrix}
\lam_1(\zz)^{h_1} & & & \\
 & \lam_2(\zz)^{h_2}  & &{\Large \mbox{$*$}}  \\
 & & \ddots & \\
 && &\\
 &  &  & \lam_{\rk\ul M}(\zz)^{h_{\rk \ul M}} 
\end{pmatrix}
\end{equation}
with respect to suitable $K\zzb$-bases of $\cP_{\ul M}$ and $M\ot_{A_K}K\zzb$.

\begin{definition}\label{defi.HPweight}
We call the integers $h_1, \ldots, h_{\rk \ul M}$ the {\it Hodge-Pink weights} of $\ul M$, 
and denote by $\HP(\ul M):=\{h_1,\ldots,h_{\rk \ul M}\}$ the multi-set  consisting  of them. 
\end{definition}

\begin{remark}\label{rem.HPneg}
The above discussion provides a $K\zzb$-basis ${\bm e}_i \in \cP_{\ul M}$ 
such that $\cP_{\ul M}=\bigoplus_{i=1}^{\rk \ul M}K\zzb \cdot {\bm e}_i$ and $\cQ_{\ul M}=\bigoplus_{i=1}^{\rk \ul M}K\zzb \cdot (\zz)^{-h_i}\cdot {\bm e}_i$.
Although  the negatives $-h_1\geq \cdots \geq -h_{\rk \ul M}$ are often called the Hodge-Pink weights of $\ul M$, 
 we adopt Definition \ref{defi.HPweight}  to simplify the description of Proposition \ref{prop.HPwt} below.
\end{remark}

\begin{remark}
By the matrix representation (\ref{eq.tau}),
we see that the Hodge-Pink weights are characterized as the jumps of the Hodge-Pink filtration, that is,  an integer $h$ is a Hodge-Pink weight of $\ul M$
if and only if 
\[
F^h H^1_\dR(\ul M,K) \supsetneq F^{h+1} H^1_\dR(\ul M,K),
\]
and in addition  the multiplicity of $h \in \HP(\ul M)$ equals to  
$\dim_{K}\left(F^h H^1_\dR(\ul M,K)/ F^{h+1} H^1_\dR(\ul M,K)\right)$. 
\end{remark}

For an effective $A$-motive $\ul M=(M,\tau_M)$ over $K$, 
since $\Cok(\tau_{M})  \cong \bigoplus_{i=1}^{\rk \ul M} K\zzb/(\zz)^{h_i}$ as $K$-vector spaces if $\HP(\ul M)=\{h_1,\ldots,h_{\rk \ul M}\}$,
we immediately obtain the next assertion:

\begin{proposition}\label{prop.HPwt}
Let $\ul M$ be an effective $A$-motive over $K$.
\begin{itemize}
\item[$(1)$] $\dim \ul M=\sum \HP(\ul M)$.
\item[$(2)$] For $h\geq 0$, $\ul M$ is of height $\leq h$ if and only if $\HP(\ul M) \subset [0,h]$.
%\item[(c)] The determinant $\det\tau_M \in A_{K}$ belongs to  $(\zz)^{\dim \ul M}\cdot K\zzb\mal$.
\end{itemize}
\end{proposition}

%%%%%%%%%%%%%%%%%%%%%%%%%%%%%%%%%%%%%%%%%%%%%%%%%%%%%%%%%%%%%%%%%%%%%%%%%%%%%%%%%%%%%
%
%   4. Galois representations arising from local shtukas
%
%%%%%%%%%%%%%%%%%%%%%%%%%%%%%%%%%%%%%%%%%%%%%%%%%%%%%%%%%%%%%%%%%%%%%%%%%%%%%%%%%%%%%

\section{Galois representations arising from local shtukas}\label{sec.local}
Throughout this section, 
let $L$ be a complete discrete valuation field containing $K$ and suppose that the residue field $k$  of $L$ is perfect, and that the conpostite $A \overset{\gamma} \ra K\subset L$ factors through the valuation ring $\cO$ of $L$ and satisfies $\Ker(A \overset{\gamma} \ra \cO \thra k)=\frp$ for some maximal ideal $\frp$ of $A$.
We set $\wh q:=q^{d_\frp}=\#\bF_\frp$.
Let  $\pi \in L$ be a uniformizer and 
denote by
 $\ord_L(\cdot)\colon L \ra \bZ\cup\{\infty\}$  the discrete valuation of $L$ normalized as $\ord_L(\pi)=1$. 
%Note that what we have in  mind as $L$ is either a completion $K_u$ of $K$ at some $u\mid \frp$ or 
%the completion $\wh{K}_u^\ur$ of the maximal unramified extension $K_u^\ur$ of $K_u$.
Choose and fix a uniformizer $z \in Q$ at $\frp$.
This allows us to identify $A_\frp=\bF_\frp\zb$ and $Q_\frp=\bF_\frp\zp$.
By continuity, the map  $\gamma \col A \ra \cO$ extends to $\gamma \col A_\frp \ra \cO$, and we set $\zeta :=\gamma(z) \in \cO$.
Let $\cO\zb$ be the power series ring over $\cO$ in ``variable'' $z$.
% so that we have $\cO\zb \cong A_\frp \wh{\ot}_{\bF_\frp}\cO$. 
Let us consider the endomorphism $\hsig$ of $\cO\zb$ determined by $\wh\sig(z)=z$ and $\wh\sig(\lam)=\lam^{\wh q}$ for $\lam \in \cO$.
%We notice that the ideal $J=J_\cO=\Ker (A_\cO \thra \cO)$ satisfies $J\cdot \cO\zb=(\zz)$ since $\cO\ot_{A_\cO}\cO\zb=\cO$.
For an $\cO\zb$-module $\wh M$, we set $\wh\sig\ast \wh M:=\wh M\ot_{\cO\zb,\wh\sig}\cO\zb$.
 
%%%%%%%%%%%%%%%%%%%%%%%%%%%%%%%%%%%%%%%%%%%%%%%%%%%%%%%%%%%%%%%
\subsection{Local shtukas and torsion local shtukas}\label{ss.shtuka}

\begin{definition}
A {\it local shtuka} over $\cO$ of {\it rank} $r$ is a pair $\wh {\ul M}=(\wh M, \tau_{\wh M})$ consisting of a free $\cO\zb$-module $\wh M$ of rank $r$ and an isomorphism $\tau_{\wh M}\col \wh\sig\ast \wh M[\zzf] \lrai \wh M[\zzf]$.
If $\tau_{\wh M}(\wh\sig\ast\wh M) \subset \wh M$, then $\wh {\ul M}$ is said to be {\it effective}, and if $(\zz)^h\wh M \subset \tau_{\wh M}(\wh\sig\ast\wh M) \subset \wh M$ for $h \geq 0$, then $\wh {\ul M}$ is said to be {\it effective of height $\leq h$}.

A {\it morphism} of local shtukas $f \col (\wh M, \tau_{\wh M}) \ra (\wh M', \tau_{\wh M'})$ over $\cO$ is a morphism of $\cO\zb$-modules $f\col \wh M \ra \wh M'$ such that the induced morphism $f \col \wh M[\zzf] \ra \wh M'[\zzf]$ satisfies  $\tau_{\wh M'}\circ \wh \sig\ast f=f \circ \tau_{\wh M}$.
It is called an {\it isogeny} if $f$ induces an isomorphism $f \col \wh M[\zf] \ra \wh M'[\zf]$ of $\cO\zp$-modules.
Note that if $f\col \ul{\wh M} \ra \ul{\wh M'}$ is an isogeny, then $f$ is injective and $\rk \ul {\wh M}=\rk\ul{\wh M'}$.
\end{definition}

\begin{lemma}\label{lemdimension}
Let $\wh {\ul M}=(\wh M, \tau_{\wh M})$ be a local shtuka over $\cO$.

\begin{itemize}
\item[{$(1)$}] There is an integer $d \in \bZ$ such that $\det \tau_{\wh M}\in (\zz)^d\cdot \cO\zb\mal$.
\item[$(2)$] If $\wh{\ul M}$ is effective, then the integer $d$ in $(1)$ satisfies $d \geq 0$ and $\wh M/\tau_{\wh M}(\wh\sig\ast \wh M)$ is a free $\cO$-module of rank $d$ which is  annihilated by $(\zz)^d$. 
\end{itemize}
\end{lemma}

\begin{proof}
This is   \cite[Lemma 3.2.3]{HK20}.
\end{proof}

\begin{definition}
Define the {\it dimension} of $\wh{\ul M}$ by the integer $d$ as in Lemma \ref{lemdimension}, and set $\dim \wh{\ul M}:=d$.
Clearly $\wh{\ul M}$ is of height $ \leq d$ if it is effective.
\end{definition}

Let $\wh{\ul M}=(\wh M,\tau_{\wh M})$ be a local shtuka over $\cO$.
Then $\tau_{\wh M}$ induces an isomorphism 
\[
\tau_{\wh M} \col \hsig\ast \wh M \ot_{\cO\zb}L^\sep\zb \lrai \wh M\ot_{\cO\zb}L^\sep\zb
\]
since $\zz$ is invertible in $L^\sep\zb$.
We define the action of $G_L=\Gal(L^\sep/L)$ on 
$\wh M \ot_{\cO\zb}L^\sep\zb$ via 
the trivial action on $\wh M$ and the natural action on
$L^\sep\zb$.
Now the $\hsig$-invariant subring  of $L^\sep\zb$ is $\bF_\frp\zb=A_\frp$.
Then we define  
the {\it dual Tate module} of $\ul{\wh M}$ by  
\[  
\wc T_\frp \wh{\ul M}:=(\wh M \ot_{\cO\zb}L^\sep\zb)^{\tau}:=
\{ m \in \wh M \ot_{\cO\zb}L^\sep\zb  \mid  \tau_{\wh M}(\hsig\ast m)=m\},
\] 
which is a free $A_\frp$-module of rank equal to $\rk \wh{\ul M}$ with a continuous $G_L$-action.
We also define the {\it rational dual Tate module} of  $\wh{\ul M}$ 
by the $Q_\frp$-vector space with a continuous $G_L$-action
$
\wc V_\frp \wh{\ul M}:= \wc T_\frp \wh{\ul M} \ot_{A_\frp}Q_\frp.
$ 
As dual of them, the {\it Tate module} and the {\it rational Tate module} of $\ul{\wh M}$ are defined by 
\[
T_\frp\wh{\ul M}=\Hom_{A_\frp}(\wc T_\frp \wh{\ul M}, A_\frp)\es\es 
\mbox{and}
\es\es
 V_\frp\wh{\ul M}=T_\frp\wh{\ul M}\ot_{A_\frp}Q_\frp.
\]

\begin{example}[{cf.\ \cite[Example 3.2.2]{HK20}}]\label{ex.motiveshtuka}
It is known that local shtukas relate with effective $A$-motives with good reduction.
To describe this,  let $\ul M$ be an effective $A$-motive over $L$ 
with good reduction.
Let $\ul \cM=(\cM, \tau_\cM)$ be a good model  of $\ul M$. 
Now we consider the $\frp$-adic completion $A_{\frp,\cO}:=A_\frp\wh\ot_{\bF_q}\cO$ of $A_\cO$ and set 
$\ul \cM\ot{A_{\frp,\cO}}:=(\cM\ot_{A_\cO}A_{\frp,\cO},\tau_\cM\ot\id$).
Then we get the  {\it local shtuka associated with $\ul M$} denoted by $\wh {\ul M}_\frp(\ul M)$ as follows.

\begin{itemize}
\setlength{\leftskip}{-10pt}
\item 
If $\bF_\frp=\bF_q$, and so $\wh q=q$ and $\wh \sig=\sig$, then we have $A_{\frp,\cO}=\cO\zb$ and $J\cdot A_{\frp,\cO}=(\zz)$.
Hence $\ul\cM\ot A_{\frp,\cO}$ itself becomes an effective local shtuka over $\cO$.
 We set 
$\wh {\ul M}_\frp(\ul M):=\ul\cM\ot A_{\frp,\cO}$. 
\item On the other hand, let us assume   $d_\frp=[\bF_\frp:\bF_q]>1$.
For each $0 \leq i \leq d_\frp-1$, we consider the ideals $\fra_i\subset A_{\frp,\cO}$ generated by $\{b\ot1-1\ot \gamma(b)^{q^i} \mid b \in \bF_\frp\}$.
Then these ideals satisfy $\prod_{i=0}^{d_\frp-1} \fra_i=(0)$ because  for any $b\in \bF_\frp$, the polynomial $\prod_{i=0}^{d_\frp-1}(X-b^{q^i})$ is a multiple of the minimal polynomial of $b$ over $\bF_q$ and even equal to it when $\bF_\frp=\bF_q(b)$.
The Chinese remainder theorem yields the decomposition 
\[
A_{\frp,\cO}=\prod_{i=0}^{d_\frp-1} A_{\frp,\cO}/\fra_i
\]
whose factors have  canonical isomorphisms $ A_{\frp,\cO}/\fra_i \cong \cO\zb$.
In addition, the factors are cyclically permuted by $\sig$ since $\sig(\fra_i)=\fra_{i+1}$, and hence $\hsig=\sig^{d_\frp}$ stabilizes each factor.
Here it follows that the  ideal $J$ decomposes as $J\cdot A_{\frp,\cO}/\fra_0=(\zz)$ and $J\cdot A_{\frp,\cO}/\fra_i=(1)$ for $i\neq 0$.
Considering the $d_\frp$-th iteration $\tau_\cM^{d_\frp}:=\tau_\cM \circ \sig\ast\tau_\cM \circ \cdots \circ \sig^{(d_\frp-1)*}\tau_\cM$, we get the effective local shtuka $\ul{\wh M}_\frp(\ul M):=(\cM\ot_{A_\cO}\left(A_{\frp,\cO}/\fra_0\right), (\tau_{\cM}\ot \id)^{d_\frp})$.
This definition coincides with the before one when $d_\frp=1$.
By construction, we see that $\Cok(\tau_\cM)$ is canonically isomorphic to $\wh M/\tau_{\wh M}(\wh\sig\ast\wh M)$ as free $\cO$-modules. 
Hence we have $\dim \ul M=\dim \wh{\ul M}_\frp(\ul M)$.
\end{itemize}
\end{example}

\begin{remark}\label{rem.localshtukaeq}
The local shtuka  $\ul{\wh M}_\frp(\ul M)$ allows to recover $\ul \cM \ot A_{\frp,\cO}$ via the isomorphism
\[
\bigoplus_{i=0}^{d_\frp-1}(\tau_\cM\otimes\id)^i\  {\rm mod}\ \fra_i \colon \left( 
\bigoplus_{i=0}^{d_\frp-1}\sig^{i*}(\cM\otimes_{A_{\cO}}A_{\frp,\cO}/\fra_0),
(\tau_\cM^{d_\frp}\otimes\id)\oplus\bigoplus_{i \neq 0}\id
\right)
\lrai
\ul\cM\otimes_{A_{\cO}}A_{\frp,\cO}.
\]
Sometimes $\ul \cM\ot_{A_\cO} A_{\frp,\cO}$ is called an effective local shtuka.
In fact, the category of locally free $A_{\frp,\cO}$-modules $N$ equipped with injective $A_{\frp,\cO}$-homomorphisms $\tau_{N}\col \sig\ast N \ra N$ satisfying $J^n\cdot\Cok(\tau_{N})=0$ for some $n>0$  is equivalent to the category of effective local shtukas over $\cO$ via the functor $(N,\tau_N)\mapsto (N/\fra_0,(\tau_N\ot\id)^{d_\frp})$;  see \cite[Propositions 8.8 and 8.5]{BH11}.
\end{remark}

The local shtuka $\wh{\ul M}_\frp(\ul M)$ associated with $\ul M$ satisfies  $T_\frp \ul M \cong T_\frp\left(\wh{\ul M}_\frp(\ul M)\right)$ by the following:

\begin{proposition}
Let $\ul M$ be an effective $A$-motive over $L$ with good reduction and set $\wh{\ul M}=\ul{\wh M}_\frp(\ul M)$.
Then there is a canonical functorial isomorphism $H^1_\frp(\ul M,A_\frp) \overset{\sim}{\ra} \wc T_\frp \wh{\ul M}$ as representations of $G_L$.
\end{proposition}

\begin{proof}
See \cite[Proposition 3.4.6]{HK20}.
\end{proof}

We next introduce a torsion version of equi-characteristic Fontaine's theory given by {\it torsion local shtukas}, which are a function field analog of 
finite flat group schemes of $p$-power order.

\begin{definition}
A {\it torsion local shtuka} over $\cO$ is a pair $\ul\frM=(\frM,\tau_\frM)$ consisting of a finitely presented $\cO\zb$-module $\frM$ which is $z$-power torsion and finite free over $\cO$, and an isomorphism $\tau_\frM \col \wh\sig\ast \frM[\zzf] \lrai \frM[\zzf]$.
If $\tau_\frM(\wh\sig\ast\frM)\subset \frM$, then $\ul \frM$ is said to be {\it effective}, and if 
$(\zz)^h\frM \subset \tau_\frM(\wh\sig\ast\frM)\subset \frM$ for $h\geq 0$, then $\ul \frM$ is said to be {\it effective of height $\leq h$}.
Define the rank  of $\ul \frM$ by $\rk\ul \frM:=\rk_\cO\frM$.

A {\it morphism} of torsion local shtukas $f\colon \ul\frM\ra \ul \frM'$ over $\cO$ is a morphism of $\cO\zb$-modules 
$f\col \frM\ra \frM'$ such that $\tau_{\frM'}\circ \wh\sig\ast f=f\circ\tau_\frM$.
\end{definition}

\begin{remark}
Note that $\frM$ injects into $\frM[\zzf]$ because $\zeta$ is $\frM$-regular and $\frM$ has $z$-power torsion.  
\end{remark}

One can associate torsion local shtukas with torsion Galois representations in the same way for local shtukas as follows.
Let $\ul{\frM}=(\frM,\tau_{\frM})$ be a torsion local shtuka over $\cO$.
Then $\tau_{\frM}$ induces an isomorphism $\tau_{\frM}\col  \hsig\ast \frM \ot_{\cO\zb}L^\sep\zb \lrai \frM\ot_{\cO\zb}L^\sep\zb$.
Define
\[
\wc T_\frp \ul\frM:=(\frM \ot_{\cO\zb}L^\sep\zb)^{\tau}:=
\{ m \in \frM \ot_{\cO\zb}L^\sep\zb  \mid  \tau_{\frM}(\hsig\ast m)=m\},
\]
which has a discrete $G_L$-action induced by that on $L^\sep\zb$.
In fact, 
 it is known by \cite[Proposition 3.7.8]{HK20} that
$\wc T_\frp \ul\frM$ is a torsion $A_\frp$-module of length equal to $\rk \ul \frM=\rk_\cO \frM$
with discrete $G_L$-action, and that 
 the inclusion $\wc T_\frp\ul\frM \hra \frM\ot_{\cO\zb}L^\sep\zb$ induces a canonical $G_L$-equivariant isomorphism of $L^\sep\zb$-modules
\begin{equation}\label{eq.toret}
\wc T_\frp\ul\frM\ot_{A_\frp}L^\sep\zb \lrai \frM\ot_{\cO\zb} L^\sep\zb
\end{equation}
such that
it commutes with the diagram
\[
\xymatrix{
\wc T_\frp\ul\frM\ot_{A_\frp}L^\sep\zb \ar[r]^-\sim\ar[d]_-{\id\ot\hsig} & \frM\ot_{\cO\zb} \ar[d] L^\sep\zb \\
\wc T_\frp\ul\frM\ot_{A_\frp}L^\sep\zb\ar[r]^-\sim & \frM\ot_{\cO\zb} L^\sep\zb ,
}
\]
where $\frM\ot_{\cO\zb} L^\sep\zb \ra \frM\ot_{\cO\zb} L^\sep\zb$ is given by 
 $m\ot \lam\mapsto \tau_{\frM}(\hsig\ast m)\ot \hsig(\lam)$ for $m \in \frM$ and $\lam \in L\zb$.
Also we define 
\[
T_\frp\ul\frM:=\Hom_{A_\frp}(\wc T_\frp\frM, Q_\frp/A_\frp).
\]
Then $\ul\frM \mapsto T_\frp\ul\frM$ yields a  contravariant exact functor from the category of torsion local shtukas to that of torsion $A_\frp$-modules with discrete $G_L$-action. 

For an isogeny of local shtukas $f\colon \ul (\wh M,\tau_{\wh M}) \ra (\wh M',\tau_{\wh M'})$ over $\cO$, set $\frM_f:=\Cok(f)$.
We see that $\tau_{\wh M'}$ induces an isomorphism $\tau_f \colon \wh\sig\ast \frM_f[\zzf] \rai \frM_f[\zzf]$.
Then we have the following (cf.\ \cite[Example 3.7.3]{HK20}): 

\begin{lemma}
Let $f\colon \ul {\wh M} \ra \ul{\wh M'}$ be an isogeny of local shtukas over $\cO$. 
Then $\ul{\frM}_f=(\frM_f,\tau_f)$ is a torsion local shtuka over $\cO$, which is effective of height $\leq h$ if 
 $\ul{\wh M'}$ is effective of height $\leq h$.
\end{lemma}

\begin{proof}
Recall that $f$ is injective and $\rk \ul{\wh M}=\rk\ul{\wh M'}$.
By the definition of isogenies, there is a positive integer $n$ such that  $z^n \wh M' \subset f(\wh M)$, and so $\frM_f$ is annihilated by $z^n$.
Tensoring with the residue field $k$ over $\cO$, we have an exact sequence 
\[
0 \ra \Tor_1^\cO(k,\frM_f) \ra \wh M'\ot_{\cO\zb}k\zb \overset{f\ot\id_k}{\lra} \wh M\ot_{\cO\zb}k\zb \ra \frM_f\ot_\cO k\ra 0.
\]
Since  $\wh M\ot_{\cO\zb}k\zb$ and $\wh M'\ot_{\cO\zb}k\zb$ are free $k\zb$-modules of rank $\rk \ul{\wh M}=\rk\ul{\wh M'}$ and $\frM_f\ot_{\cO}k$ is annihilated by $z^n$, the elementary divisor theorem implies that $f\ot\id_k$ is injective.
Hence $\Tor_1^\cO(k,\frM_f)=0$.
Applying the snake lemma to the diagram 
\[
\xymatrix{
0 \ar[r]& \wh M\ar[r]^-f\ar[d]^-{z^n} & \wh M' \ar[r]\ar[d]^-{z^n}& \frM_f \ar[r]\ar[d]^{z^n}& 0 \\
0 \ar[r]& \wh M\ar[r]^-f & \wh M' \ar[r]& \frM_f \ar[r]& 0,
}
\]
we see that  $\frM_f=\Cok(f\col \wh M/z^n\wh M \ra \wh M'/z^n\wh M')$ and hence $\frM_f$ is finitely presented over $\cO$.
Thus Nakayama's lemma shows that $\frM_f$ is finite free over $\cO$.
If $\ul{\wh M'}$ is effective of height $\leq h$, then $\ul\frM_f$ is effective. 
The surjection $\wh{M'}/\tau_{\wh M'}(\wh\sig\ast {\wh M'}) \ra \frM_f/\tau_f(\wh\sig\ast \frM_f)$ shows that $\ul{\frM}_f$ is of height $\leq h$.
\end{proof}

\begin{remark}
Conversely, it is known that any torsion local shtuka $\ul{\frM}$ over $\cO$ is of the form $\ul{\frM}=\ul{\frM}_f$ for  some isogeny $f\colon \ul {\wh M} \ra \ul{\wh M'}$ of local shtukas over $\cO$; see \cite[Lemma 3.7.5]{HK20}.
\end{remark}

\begin{example}
For a local shtuka  $\ul {\wh M}=(\wh M,\tau_{\wh M})$ over $\cO$ and 
any positive integer $n$, the map $f_n\colon {\wh M} \ra {\wh M}; m\mapsto z^nm$  gives an isogeny $f_n \colon \ul {\wh M} \ra \ul {\wh M}$.
Thus we get the torsion local shtuka $\ul{{\wh M}}/z^n\ul{{\wh M}}:=({\wh M}/z^n{\wh M},\tau_{\wh M}\pmod {z^n})$.
\end{example}

\begin{lemma}\label{lem.modz}
Let $f\colon \ul {\wh M} \ra \ul{\wh M'}$ be an isogeny of local shtukas over $\cO$.
Then for the torsion local shtuka $\ul{\frM}_f $ induced by $f$, we have $T_\frp\ul{\frM}_f \cong \Cok(T_\frp f)$ as torsion $A_\frp$-modules with discrete $G_L$-action, where $T_\frp f\colon T_\frp\ul{\wh M'} \ra T_\frp\ul{\wh M}$ is the map induced by $f$.
In particular, for a local shtuka $\ul{\wh M}$,   we have a natural isomorphism $T_\frp\ul{\wh M}\ot_{A_\frp}\left(A_\frp/\frp^nA_\frp\right) \cong T_\frp\left(\ul{\wh M}/z^n\ul{\wh M}\right)$.
\end{lemma}

\begin{proof}
Applying   \cite[Lemma 5.1.9]{Ki09} to the map $f\colon \wh M\ot_{\cO\zb}L\zb \ra \wh M'\ot_{\cO\zb}L\zb$ induced by $f$, we obtain the conclusion.
\end{proof}

\begin{proposition}\label{prop.torcat}
Let $\ul\frM$ be a torsion local shtuka over $\cO$.
Suppose that there is an exact sequence 
\[
0\lra T' \lra T_\frp\ul\frM \lra T''\lra 0
\]
of torsion $A_\frp$-modules of finite length with discrete $G_L$-action. 
Then the sequence is induced by an exact sequence 
\[
0\lra \ul\frM'' \lra \ul\frM \lra \ul\frM' \lra 0
\]
of torsion local shtukas over $\cO$ via  the functor $\ul\frM \mapsto T_\frp\ul\frM$.
Moreover, if $\ul\frM$ is effective of height $\leq h$, then so are $\ul\frM'$ and $\ul\frM''$.
\end{proposition}
\begin{proof}
This is \cite[Proposition 3.7.15]{HK20} (or \cite[Proposition 9.2.2]{Ki09}), and so we just show how to give $\ul\frM'$ and $\ul\frM''$.
We set  $\wc T=\wc T_\frp\ul\frM$, $\wc T'=\Hom_{A_\frp}(T', Q_\frp/A_\frp)$, and $\wc T''=\Hom_{A_\frp}(T'', Q_\frp/A_\frp)$.
Taking $G_L$-fixed parts of the inverse of (\ref{eq.toret}), we have an isomorphism 
\[
\frM\ot_{\cO\zb}L\zb \lrai (\wc T\ot_{A_\frp}L^\sep\zb)^{G_L}.
\]
As the composition with $ (\wc T\ot_{A_\frp}L^\sep\zb)^{G_L}\thra  (\wc T'\ot_{A_\frp}L^\sep\zb)^{G_L}$, we obtain a surjection 
\[
f\colon \frM\ot_{\cO\zb}L\zb \thra (\wc T'\ot_{A_\frp}L^\sep\zb)^{G_L}.
\]
Set $\frM'=f(\frM)$.
Then $\tau_\frM$ induces an isomorphism $\tau_{\frM'}\colon \hsig\ast\frM'[\zzf]\rai \frM'[\zzf]$ and the pair $\ul\frM'=(\frM',\tau_{\frM'})$
is a torsion local shtuka over $\cO$ with $T_\frp\ul\frM' \cong T'$.
Next let $\frM''$ be the kernel of $f|_\frM\col \frM \thra \frM'$ and $\tau_{\frM''}=\tau_{\frM}|_{\frM''}$.
Then $\ul\frM''=(\frM'',\tau_{\frM''})$ is a torsion local shtuka over $\cO$ with $T_\frp\ul\frM'' \cong T''$.
\end{proof}

%%%%%%%%%%%%%%%%%%%%%%%%%%%%%%%%%%%%%%%%%%%%%%%%%%%%%%%%%%%%%%%
\subsection{Tame inertia weights}

Let $I_L \subset G_L$ be the inertia subgroup, so that $I_L\cong \Gal(L^\sep/L^\ur)$.
Let $I_L^\wild \subset I_L$ be the wild inertia subgroup, that is, the maximal pro-$p$ subgroup of $I_L$.
Define the {\it tame inertia group} by $I_L^\tame:=I_L/I_L^\wild\cong \Gal(L^\tame/L^\ur)$, where
$L^\tame$ is the maximal tamely ramified extension of $L^\ur$.  
We note that $I_L^\tame$ is pro-cyclic.

Let $\bF$ be an intermediate finite subfield of $k/\bF_\frp$ and set $r=[\bF:\bF_\frp]$.
Recall that we now set $\wh q=\#\bF_\frp$, so that $\#\bF={\wh q}^r$.
Denote by $\mu_{{\wh q}^r-1}(L)$ the group of $({\wh q}^r-1)$-st roots of unity in $L$ and choose an isomorphism $\mu_{{\wh q}^r-1}(L)\rai \bF\mal$. 
Let  $\pi \in L$ be a uniformizer.
Denote by $L_r^{(\pi)}$ the finite Galois extension of $L$ generated by the roots of $X^{{\wh q}^r-1}-\pi$.
Then $L_r^{(\pi)}/L$ is a totally tamely ramified extension whose ramification index is ${\wh q}^r-1$.
For a root $\pi_r \in L_r^{(\pi)}$ of $X^{{\wh q}^r-1}-\pi$, we obtain the surjective character
\[
\theta_r=\theta_{r/L}\col I_L \thra \mu_{{\wh q}^r-1}(L)\rai \bF\mal;\es\es s\mapsto \frac{s(\pi_r)}{\pi_r},
\]
which is called a {\it fundamental character of level $r$} (cf.\ \cite{Ser72}).
It factors through $I_L^\tame$ and does not depend on the choice of $\pi$.
For a field extension $L'/L$ with finite ramification index $e_{L'/L}$, we have 
$
\theta_{r/L}|_{I_{L'}}=(\theta_{r/L'})^{e_{L'/L}}
$ by construction.

Let  $\psi\col I_L\ra \bF\mal$ be an $\bF\mal$-valued character. 
Then it factors through $I_L^\tame$ and there is an integer $0 \leq n \leq {\wh q}^r-1$ such that $\psi=\theta_r^n$. 
If the integer $n$  decomposes  $n=n_0{\wh q}^{r-1}+n_1{\wh q}^{r-2}+\cdots+n_{r-1}$ with integers $0\leq n_i \leq \wh q-1$, then  $n_0,\ldots,n_{r-1}$ are independent of the choice of $\mu_{{\wh q}^r-1}(L)\rai \bF\mal$.
\begin{definition}
We call these integers $n_0,\ldots,n_{r-1}$ the {\it tame inertia weights} of $\psi\col I_L\ra \bF\mal$ and denote by $\TI(\psi)=\{n_0,\ldots,n_{r-1}\}$ the multi-set consisting of them.
\end{definition}

To define the tame inertia weights of $\frp$-adic representations of $G_L$, we need the following:

\begin{lemma}\label{lem.fund}
An irreducible $\bF_\frp$-representation $\bar V$ of $I_L$ becomes a one-dimensional representation of $I_L^\tame$ over a finite extension $\bF$ of $\bF_\frp$ with $[\bF:\bF_\frp]=\dim_{\bF_\frp}\bar V$. 
\end{lemma}

\begin{proof}
By irreducibility, the
$I_L^\wild$-action on $\bar V$ is trivial and so 
 the action of $I_L$ factors through $I_L^\tame$.
 Thus we obtain a group homomorphism $\rho\col I_L^\tame \ra \GL_{\bF_\frp}(\bar V)$.
Let us consider the commutant  $\bF:=\End_{\bF_\frp[I_L^\tame]}(\bar V) \subset \GL_{\bF_\frp}(\bar V)$.
Then 
Schur's lemma implies that $\bF$ is a finite division algebra, so it is a finite field.
Since $I_L^\tame$ is abelian, the image $\rho(I_L^\tame)$ is contained in the commutant $\bF$.
Thus we see by irreducibility that  $\bar V$ is  one-dimensional over  $\bF$ and so $[\bF:\bF_\frp]=\dim_{\bF_\frp}\bar V$.
\end{proof}

Let $V$ be a $\frp$-adic representation of $G_L$.
Take a $G_L$-stable $A_\frp$-lattice $T \subset V$ and consider the $\bF_\frp$-representation $\bar V:=T/\frp T$ of $G_L$.
Then by the Brauer-Nesbitt theorem,  the isomorphism class of the semi-simplification $\bar V^\ss=\bigoplus_i \bar V_i$ is independent of the choice of $T$.
We see that each simple factor $\bar V_i$ is irreducible as a representation of $I_L$ because the $I_L$-fixed part ${\bar V}_i^{I_L}$ is a subrepresentation of $\bar V_i$.
Hence Lemma \ref{lem.fund} implies that $I_L^\tame$-action on $\bar V^\ss$ is described by 
a sum of characters $\psi_i\col I_L^\tame \ra \bF_i\mal$, where $\bF_i$ are some finite extensions of $\bF_\frp$. 
Replacing $L$ by a suitable finite unramified extension, we may assume that each $\bF_i$ can be embedded into $k$ and so
each $\psi_i$ is a power of the fundamental character $\theta_{r_i}:I_L^\tame \ra \bF_i\mal$ of level $r_i=[\bF_i:\bF_\frp]$.

\begin{definition}
We define the multi-set of {\it tame inertia weight} $\TI(V)$ of $V$ 
by the union  of the $\TI(\psi_i)$ for all $i$.
\end{definition}
%%%%%%%%%%%%%%%%%%%%%%%%%%%%%%%%%%%%%%%%%%%%%%%%%%%%%%%%%%%%%%%
\subsection{Torsion local shtukas with coefficients} 

To study tame inertia weights of Galois representations coming from local shtukas, we shall give an overview of the theory on torsion local shtukas with ``coefficients''. 
Note that we only consider the finite field coefficients case; see \cite{Ki09} and \cite{HK20} in general.

As before,  let $\bF$ be an intermediate finite field of $k/\bF_\frp$ with degree $r=[\bF:\bF_\frp]$.
We regard $\bF$ as an $A_\frp$-algebra via $A_\frp=\bF_\frp\zb \thra \bF_\frp \subset \bF; z \mapsto 0$.
Let denote  $\cO\zb_\bF:=\cO\zb\ot_{A_\frp}\bF$ and define $\wh \sig \col \cO\zb_\bF \ra \cO\zb_\bF$ by 
$\bF$-linearly extending $\wh \sig \col \cO\zb \ra \cO\zb$.

\begin{definition}
A {\it torsion local $\bF$-shtuka} over $\cO$ is a pair $\ul\frM_\bF=(\frM_\bF,\tau_{\frM_\bF})$ consisting of a free $\cO\zb_\bF$-module of finite rank  and an $\cO\zb_\bF$-isomorphism $\tau_{\frM_\bF}\col \wh\sig\ast\frM_\bF[\zzf]\rai \frM_\bF[\zzf]$.
By {\it $\cO\zb_\bF$-rank} of $\ul\frM_\bF$, we mean the rank of the $\cO\zb_\bF$-module $\frM_\bF$.
We consider the obvious notion of morphisms.

We may regard $\ul\frM_\bF$ as a torsion local shtuka over $\cO$ by forgetting the $\bF$-action.
We say that $\ul\frM_\bF$ is {\it effective} (resp.\ {\it effective of height $\leq h$} for some $h\geq 0$) if it is so as a torsion local shtuka.
\end{definition}

\begin{remark}
By  \cite[Lemma 3.7.24]{HK20}, any torsion local shtuka $\ul\frM=(\frM,\tau_{\frM})$ equipped with a $\tau_\frM$-compatible $\bF$-action is a torsion local $\bF$-shtuka, that is, the underlying module $\frM$ is free over $\cO\zb_\bF$. 
\end{remark}

\begin{lemma}
For a torsion local $\bF$-shtuka $\ul\frM_\bF$ over $\cO$, we see that $T_\frp\ul\frM_\bF$ becomes an $\bF$-vector space of dimension equal to the $\cO\zb_\bF$-rank of $\ul\frM_\bF$ such that the $\bF$-action on $T_\frp\ul\frM_\bF$ is induced by that on $\ul\frM_\bF$. 
\end{lemma}

\begin{proof}
This is  \cite[Lemma 3.7.23 (a)]{HK20}.
\end{proof}

For a uniformizer $\pi \in L$,
we identify $L=k\dpl\pi\dpr$ and 
$\cO=k\llbracket \pi \rrbracket$, and fix an embedding $\bF\hra k \subset \cO$.
Set $e:=\ord_L(\zeta)$.
Canonically extending $\gamma\col A\subset Q \ra K$ to the embedding $\gamma\col Q_\frp \ra L$, we see that the integer $e$ is the ramification index of $L/\gamma(Q_\frp)$.  
Let $h\geq 0$.
For any $\lam \in k\mal$ and any $r$-tuple $\ul n=(n_0,n_1,\ldots,n_{r-1})$ of integers with $0 \leq n_i \leq eh$, 
we obtain an important effective torsion local $\bF$-shtuka $\ul\frM_{\SSC(\lam,\ul n)}=(\frM_{\SSC(\lam,\ul n)},\tau)$ of $\cO\zb_\bF$-rank one  as follows.
We view $\cO$ as an $\cO\zb$-algebra via $z\mapsto 0$.
We set 
\[
\frM_{\SSC(\lam,\ul{n})}:=\bigoplus_{i=0}^{r-1}\cO\cdot {\bm e}_i
\]
% a free $\cO$-module with a basis ${\bm e}_i$, so that it is finitely presented over $\cO\zb$. 
and define $\tau\col \hsig\ast \frM_{\SSC(\lam,\ul n)} \ra \frM_{\SSC(\lam,\ul n)}$ by
\begin{align*}
\tau(\hsig\ast{\bm e}_i) &:= \pi^{n_i}{\bm e}_{i+1} \es\es\es \mbox{if}\ i \neq r-1,\\
\tau(\hsig\ast{\bm e}_{r-1}) &:=\lam\pi^{n_{r-1}}{\bm e}_0.
\end{align*} 
Since $\Cok(\tau)\cong \bigoplus_{i=0}^{r-1}\cO/(\pi^{n_i})$ and 
 $(\zz)^h \equiv \xi\pi^{eh}\pmod z$ for some $\xi \in \cO\mal$, we see that $\ul\frM_{\SSC(\lam,\ul n)}$ is an effective torsion local shtuka over $\cO$ of height $\leq h$.
Let us consider the $\bF$-action on $\frM_{\SSC(\lam,\ul n)}$ given by $b\cdot {\bm e}_i:=b^{\wh{q}^i}{\bm e}_i$ for $b \in \bF$.
Then it is compatible with $\tau$ and hence $\ul\frM_{\SSC(\lam,\ul n)}$ is a torsion local $\bF$-shtuka whose $\cO\zb_\bF$-rank is one.
Thus the $G_L$-action on $T_\frp\ul\frM_{\SSC(\lam,\ul n)}$ is given by  an $\bF\mal$-valued character $\psi\col G_L\ra \bF\mal$.

\begin{lemma}\label{lem.tametor}
Let $\lam \in k\mal$ and $\ul n=(n_0,\ldots,n_{r-1})$ be  as above.
Then the character  $\psi\col G_L \ra \bF\mal$ determined by $T_\frp\ul\frM_{\SSC(\lam,\ul n)}$ satisfies $\psi|_{I_L}=\theta_r^n$, where 
$n=n_0{\wh q}^{r-1}+n_1{{\wh q}^{r-2}}+\cdots+n_{r-1}$. 
Therefore, 
if $eh \leq \wh q-1$, then we have $\TI(T_\frp\ul\frM_{\SSC(\lam,\ul n)})=\{n_0,\ldots,n_{r-1}\}$ as multi-sets.
\end{lemma}
\begin{proof}
It suffices to check that the $I_L$-action on the dual $\wc T_\frp\ul\frM_{\SSC(\lam,\ul n)}$ of  $T_\frp\ul\frM_{\SSC(\lam,\ul n)}$ is given by $\theta_r^{-n}$ for $n=\sum_{i=0}^{r-1}n_i{\wh q}^{r-1-i}$.
Since $\frM_{\SSC(\lam,\ul n)}\ot_{\cO\zb}L^\sep\zb = \bigoplus_{i=0}^{r-1}L^\sep\cdot {\bm e}_i$, any element $m \in \wc T_\frp\ul\frM_{\SSC(\lam,\ul n)}$ is of the form $m=\sum_{i=0}^{r-1}x_i{\bm e}_i$ such that $x_i \in L^\sep$ satisfy $x_i^{{\wh q}}=(\pi^{n_i})\inv x_{i+1}$ for $i\neq r-1$ and 
$x_{r-1}^{\wh q}=(\lam\pi^{n_{r-1}})\inv x_0$.
Since giving such $m$ is equivalent to giving  $x_0 \in L^\sep$ with 
$x_0^{{\wh q}^r}=(\lam \pi^n)\inv x_0$, we can regard $\wc T_\frp\ul\frM_{\SSC(\lam,\ul n)}$ as an $\bF$-submodule  of $L^\sep$ via the map 
$m=\sum_{i=0}^{r-1}x_i{\bm e}_i \mapsto x_0$.
Then the $G_L$-action on $\wc T_\frp\ul\frM_{\SSC(\lam,\ul n)}$ is given by $s(x_0)=\varepsilon(s)\theta_r(s)^{-n}x_0$ for $s\in G_L$, where $\varepsilon\col G_L\ra \bF\mal$ is the unramified character determined by $s(\lam')=\varepsilon(s)\lam'$ for a $({\wh q}^r-1)$-st root $\lam'$ of $\lam\inv$. 
This provides the conclusion.
\end{proof}

\begin{example}\label{ex.rankone}
Let $\ul{\wh M}$ be an effective local shtuka over $\cO$ with $\rk \ul{\wh M}=1$.
Then it is of the form $\ul{\wh M}=(\cO\zb\cdot {\bm e}, \tau_{\wh M}(\hsig\ast{\bm e})=\xi(\zz)^d{\bm e})$, where $d=\dim \ul M$ and  $\xi \in \cO\zb\mal$.
Since $\cO\zb\mal=\cO\mal + z\cO\zb$, the effective torsion local shtuka $\ul\frM=\ul{\wh M}/z\ul{\wh M}$ is given by $\frM=\cO\cdot {\bm e}$
and $\tau_{\frM}(\hsig\ast{\bm e})=\lambda\pi^{ed}$ for some $\lam \in \cO\mal$.
Hence, if $ed\leq \wh q-1$, then  
$\ul\frM=\ul \frM_{\SSC (\lam,\ul n)}$ for $\ul n=(ed)$ and so the $I_L$-action on $T_\frp\ul{\frM}$ is given by 
$\theta_1^{ed}$.
\end{example}

\begin{proposition}\label{prop.fshtuka}
Let $\bF \subset k$ be as above.
Then 
any effective torsion local $\bF$-shtuka $\ul\frM_\bF$ over $\cO$ of height $\leq h$ with $\cO\zb_\bF$-rank one is isomorphic to  $\ul\frM_{\SSC(\lam,\ul n)}$ for a unique $\lam\in k\mal$ up to $(k\mal)^{{\wh q}^r-1}$-multiple, and a unique $\ul n=(n_0,n_1,\ldots,n_{r-1})$  with $0 \leq n_i \leq eh$.
\end{proposition}

\begin{proof}
This is \cite[Corollary 9.1.4]{Ki09}.
\end{proof}

Based on the above preparation,  we can get an important formula for tame inertia weights of  $\frp$-adic representations arising from effective local shtukas under the condition $eh < \wh q-1$.
%Let $\ul {\wh M}$ be an effective local shtuka over $\cO$ of height $\leq h$ and write 
%\[
%\TI(\ul{\wh M}):=\TI(T_\frp\ul{\wh M}\ot_{A_\frp}\bF_\frp)
%\]
%for the multi-set of tame inertia weights of the $\bF_\frp$-representation $T_\frp\ul{\wh M}\ot_{A_\frp}\bF_\frp$.

\begin{proposition}\label{prop.localtame}
Let $\ul {\wh M}$ be an effective local shtuka over $\cO$ of height $\leq h$.
Suppose that the ramification index $e=\ord_{L}(\zeta)$ of $L/\gamma(Q_\frp)$ satisfies $eh < \wh q-1$.
Then $\TI(V_\frp\ul{\wh M}) \subset [0,eh]$ and the equation 
\[
e\cdot \dim\ul{\wh M}=\Sigma  \TI(V_\frp\ul{\wh M})
\] 
holds.
\end{proposition}

\begin{proof}
We may assume that $k$ is algebraically closed and $G_L=I_L$ for the following reason.
Let $\wh L^\ur$ be the completion of the maximal unramified extension $L^\ur$ of $L$ and denote by $\wh \cO^\ur$ the  valuation ring of $\wh L^\ur$. 
Then $\wh L^\ur$ has residue field $k^\sep$ and $G_{\wh L^\ur} \cong I_L$.
In addition, 
it follows that $\ul{\wh M}\ot\wh\cO^\ur\zb:=(\wh M\ot_{\cO\zb}\wh\cO^\ur\zb, \tau_{\wh M}\ot\id)$ becomes an effective local shtuka over $\wh \cO^\ur$ which has the same rank and dimension of $\ul{\wh M}$, and that 
$T_\frp\ul{\wh M}|_{I_L}$ is canonically isomorphic to $T_\frp({\ul {\wh M}}\ot\wh\cO^\ur\zb)$; see \cite[Proposition 8.1.13]{Ki09} for more details.

Now we  set $r=\rk\ul{\wh M}$ and $d=\dim\ul{\wh M}$, so that $\det\tau_{\wh M} \in (\zz)^d\cdot \cO\zb\mal$ by Lemma \ref{lemdimension}.
Consider the effective torsion local shtuka $\ul\frM:=(\frM,\tau_{\frM}):=\ul{\wh M}/z\ul{\wh M}$. 
Then  $\ul\frM$ is  of height $\leq h$. 
Since $T_\frp\ul{\wh M}\ot_{A_\frp}\bF_\frp \cong T_\frp\ul\frM$ as $\bF_\frp$-representations of $G_L$, we have $\TI(V_\frp\ul{\wh M})=\TI(T_\frp\ul{\frM})$.
Since $\det\tau_{\frM}\equiv\det\tau_{\wh M}\pmod z$ and $\cO\zb\mal=\cO\mal + z\cO\zb$, we have   
$\ord_L(\det\tau_{\frM})=ed$.
Thus it suffices to check the equation $\ord_L(\det\tau_{\frM})=\Sigma\TI(T_\frp\ul\frM)$.

To do this,
we first consider the case where $T_\frp\ul\frM$ is irreducible.
Then Lemma \ref{lem.fund} implies that $T_\frp\ul\frM$ is one-dimensional representation of $I_L$ over the  finite field $\bF:=\End_{\bF_\frp[I_L^\tame]}(T_\frp\ul\frM)$ of degree $[\bF:\bF_\frp]=r$.
Taking $G_L$-invariant parts of both sides of (\ref{eq.toret}), we get an isomorphism 
\[
(\wc T_\frp\ul\frM \ot_{A_\frp}L^\sep\zb)^{G_L} \rai \frM\ot_{\cO\zb}L\zb.
\]  
Via this isomorphism, the $\bF$-action of $\wc T_\frp\ul\frM$ induces an $\bF$-action on $\frM\ot_{\cO\zb}L\zb$, which commutes with $\tau_{\frM}\ot\id$.
By the assumption $eh<\wh q-1$, it follows by \cite[Proposition 9.3.3.\ (2)]{Ki09}
that this $\bF$-action preserves $\frM$ and so  $\ul\frM$ becomes an  effective torsion local $\bF$-shtuka.
Since $\ul\frM$ is of $\cO\zb_\bF$-rank one, it follows by Proposition \ref{prop.fshtuka} that $\ul\frM$ is isomorphic to $\ul\frM_{\SSC(\lam,\ul n)}$ for some $\lam \in k\mal$ and $\ul{n}=(n_0,\ldots,n_{r-1})$ with $0 \leq n_i \leq eh$.
Hence 
Lemma \ref{lem.tametor} shows that $\TI(T_\frp\ul\frM)=\{n_0,\ldots,n_{r-1}\}$.
Since $\Cok(\tau_{\frM})\cong \bigoplus_{i=0}^{r-1}\cO/(\pi^{n_i})$,
 %by the definition of $\ul\frM_{\SSC(\lam,\ul n)}$, 
 we obtain 
$\ord_L(\det\tau_{\frM})=n_0+\cdots+n_{r-1}=\Sigma\TI(T_\frp\ul{\frM})$.

Next, we suppose that $T_\frp\ul\frM$ is not irreducible.
Then there is  a descending filtration by $G_L$-stable $\bF_\frp$-subspaces of $T_\frp\ul\frM$
\[
0=T_0 \subset T_1 \subset \cdots \subset T_n=T_\frp\ul\frM
\]
 by the Jordan-H\"{o}lder theorem, and consider the semi-simplification $(T_\frp\ul\frM)^\ss=\bigoplus_{j=1}^nT_j/T_{j-1}$. 
Then
$\Sigma\TI(T_\frp\ul\frM)$ is the sum of all $\Sigma\TI(T_j/T_{j-1})$.
Proposition \ref{prop.torcat} shows that all sequences $0\ra T_{j-1}\ra T_j \ra T_j/T_{j-1}\ra 0$ come from exact sequences of  effective torsion local shtukas of height $\leq h$
\[
0\lra \ul\frM_{j} \lra \ul\frN_j \lra \ul\frN_{j-1}\lra 0
\]
with $T_\frp\ul\frM_j=T_{j}/T_{j-1}$ and $T_\frp\ul\frN_j=T_j$.
Since $\det\tau_{\frN_j}=(\det\tau_{\frM_j}) \cdot (\det\tau_{\frN_{j-1}})$ and $\det\tau_{\frN_n}=\det\tau_{\frM}$, we have 
$\det\tau_\frM=\prod_{j=1}^n \det\tau_{\frM_j}$.
Therefore $\ord_L(\det\tau_{\frM})=\sum_{j=1}^n \ord_L(\det\tau_{\frM_j})$.
Since each $T_\frp\ul\frM_j$ is irreducible, we get $\ord_L(\det_{\frM_j})=\Sigma\TI(T_j/T_{j-1})$ for any $1\leq j \leq n$.
Hence $\ord_L(\det\tau_{\frM})=\Sigma\TI(T_\frp\ul\frM)$ holds.
\end{proof}

\begin{remark}
The relation $\TI(V_\frp\ul{\wh M}) \subset [0,eh]$ is of course true even if $eh=\wh q-1$, but 
 the equation $e\cdot \dim\ul{\wh M}=\Sigma\TI(V_\frp\ul{\wh M})$ does not hold in general.
For example, if $e=\wh q-1$ and $\rk\ul{\wh M}=\dim\ul{\wh M}=1$ (hence, $\ul{\wh M}$ is of height $\leq 1$), then $I_L$ acts on $T_\frp\ul{\wh M}\ot_{A_\frp}\bF_\frp$ via $\theta_1^e=\theta_1^{\wh q-1}=\theta_1^0$.  
Thus $\Sigma\TI(V_\frp\ul{\wh M})=0$.
\end{remark}

%%%%%%%%%%%%%%%%%%%%%%%%%%%%%%%%%%%%%%%%%%%%%%%%%%%%%%%%%%%%%%%%%%%%%%%%%%%%%%%%%%%%%
%
%   5. Rigid analytic aspects  of  effective $A$-motives 
%
%%%%%%%%%%%%%%%%%%%%%%%%%%%%%%%%%%%%%%%%%%%%%%%%%%%%%%%%%%%%%%%%%%%%%%%%%%%%%%%%%%%%%%

\section{Rigid analytic aspects  of  effective $A$-motives}

In this section, we introduce the notion of {\it strongly semi-stable reduction} for effective $A$-motives and prove a formula between tame inertia weights and dimensions of such effective $A$-motives.
We fix a maximal ideal $\frp$ of $A$.
Let $L$ be a completion of $K$ at a finite place lying above $\frp$.
Let us denote its valuation ring by $\cO$ and its residue field by $k$. 
Thus the $A$-field structure $\gamma\col A \ra K \subset L$ factors through $\cO$ and we have $\frp=\Ker(A\overset{\gamma}{\ra} \cO\thra k)$.
We take a uniformizer $\pi \in L$.

%%%%%%%%%%%%%%%%%%%%%%%%%%%%%%%%%%%%%%%%%%%%%%%%%%%%%%%%%%%%%%%
\subsection{Analytic $A(1)$-motives}
To introduce ``analytification'' of effective $A$-motives, let us first prepare some rigid analytic rings (cf.\ \cite{HH16} and \cite{Hus13}).
Recall that $A_L$ and $A_k$ are Dedekind domains.
By the isomorphism $A_\cO/\pi A_\cO \cong A_k$, the uniformizer $\pi$  is a prime element of $A_\cO$. 
\begin{definition}
Denote by $\wt A_\cO:=A_{\cO,\pi}$ the $\pi$-adic completion of $A_\cO$, and set 
$\wt A_L:=A_{\cO,\pi}[\frac{1}{\pi}]$.
\end{definition}
It follows that the topological $\cO$-algebra $\wt A_\cO$ is {\it admissible} in the sense of Raynaud, that is, it is of topologically finite presentation and has no $\pi$-torsion; see \cite[\S 7.2, Definition 3]{Bos14} for instance.
(It is also known from \cite[Proposition 1.7]{Hus13} that $\wt A_\cO$ is a regular integral domain.)
Hence $\wt A_L$ is a reduced affinoid $L$-algebra.
Let $(\frp,\pi)$ denote the ideal of $A_\cO$ generated by $\frp$ and $\pi$.
By \cite[Lemma 2.3]{HH16} (or \cite[Lemma 1.16]{Hus13}),  the $(\frp, \pi)$-adic completion of $A_\cO$ is canonically isomorphic to $A_{\frp,\cO} (=A_\frp\wh\ot_{\bF_q}\cO)$ as $A_\cO$-algebras.
Since the canonical map $A_\cO \ra A_{\frp,\cO}$ uniquely factors through $\wt A_\cO$, 
the induced map $\wt A_\cO \ra A_{\frp,\cO}$ identifies $A_{\frp,\cO}$ with the $(\frp,\pi)$-adic completion of $\wt A_\cO$, and hence it is flat.
Thus we obtain the  commutative diagram
\[
\xymatrix{
A_\cO \ar@{=}[d]\ar[r]& \wt A_\cO \ar[d]\ar[r]& \wt A_L\ar[d] \\
A_\cO \ar[r]& A_{\frp,\cO}\ar[r] & A_{\frp,L},
}
\]  
where all arrows are injective and flat.

Let us next consider the lifting of the Frobenius map $\sig=\id_A\ot(\cdot)^q$ of $A_\cO$ to the above rings. 
We see that the map $\sig$ is $\pi$-adically and $\frp$-adically continuous, and therefore it canonically extends to $\sig\col \wt A_\cO\ra \wt A_\cO$ and $\sig\col A_{\frp,\cO} \ra A_{\frp,\cO}$. 
Since $\sig\col A_\cO \ra A_\cO$ is finite flat by \cite[Remark 3.2]{HH16}, we get 
the commutative diagram
\[
\xymatrix{
A_\cO \ar[d]_-\sig\ar[r]& \wt A_\cO \ar[d]_\sig\ar[r]& A_{\frp,\cO}\ar[d]_\sig \\
A_\cO \ar[r]& \wt A_\cO\ar[r] & A_{\frp,\cO}
}
\]
satisfying that both squares are co-Cartesian and all vertical arrows are finite flat; see \cite[Lemma 3.1]{HH16}.

Now let $\wt J_\cO$ be the ideal of $\wt A_\cO$ generated by $\{a\ot 1-1\ot \gamma(a) \mid a\in A\} \subset A_\cO$, so that $\wt J_\cO=J_\cO\wt A_\cO$.
Since $A_\cO$ is Noetherian and $\cO$ is $\pi$-adically complete,  the  map $A_\cO \thra \cO; a\ot\lam \mapsto \gamma(a)\lam$ induces a surjection $\wt A_\cO \thra \cO$  and hence $\wt J_\cO$ is the kernel of this map. 
Consequently, the ideal $\wt J=J\wt A_L$ of $\wt A_L$ is a maximal ideal with residue field $L$.

\begin{definition}
An {\it analytic $A(1)$-motive over $L$ of rank $r$ and dimension} $d$ is a pair $\wt {\ul M}=(\wt M,\tau_{\wt M})$  consisting of a locally free $\wt A_L$-module of rank $r$ and an injective $\wt A_L$-homomorphism $\tau_{\wt M}\col \sig\ast \wt M \ra \wt M$ such that $\Cok(\tau_{\wt M})$ is a $d$-dimensional $L$-vector space annihilated by a power of $\wt J$. 
Set $\rk \wt{\ul M}=r$ and $\dim\wt{\ul M}=d$.
For $h\geq 0$, we say that $\wt{\ul M}$ is {\it of height $\leq h$} if $\wt J^h\cdot \Cok(\tau_{\wt M})=0$.

A {\it morphism} $f\colon \wt{\ul M} \ra \wt{\ul N}$ between analytic $A(1)$-motives over $L$ is an $\wt A_L$-homomorphism $f\colon \wt M \ra \wt N$ such that 
$f\circ \tau_{\wt M} = \tau_{\wt N} \circ \sig\ast f$.
\end{definition}

\begin{remark}
The prefix ``$A(1)$-'', which is used in \cite{HH16} to reflect the notation in \cite{BH07}, indicates that analytic $A(1)$-motives are variants of effective $A$-motives over the rigid analytic ``unit disc'' in $\Spec A_L$. 
To elaborate, let us recall a geometric interpretation of the rings $\wt A_\cO$ and $\wt A_L$.
Since $\Spec A_L$ is of finite type over $L$, there is the associated rigid analytic  space 
$(\Spec A_L)^\an$ over $L$ via the rigid analytic GAGA functor (cf.\ \cite[\S\S 5.4]{Bos14}).
On the other hand, the formal completion of the $\cO$-scheme $\Spec \cO$ along its special fiber is 
the formal $\cO$-scheme $\Spf \wt A_\cO$.  
Since $\wt A_\cO$ is admissible as we have seen, its associated rigid analytic space $(\Spf \wt A_\cO)^\rig$, which is so-called the {\it Raynaud generic fiber}, coincides with the affinoid $L$-space $\Sp \wt A_L$.
Thus it can be regarded as the unit disc of the radius of convergence one in $(\Spec A_L)^\an$.

In contrast,  another (and  well-known) analytic variant of effective $A$-motives is the notion of {\it analytic $\tau$-sheaves}, which are locally free $\cO_{(\Spec A_L)^\an}$-modules $\cF$ of finite rank endowed with  suitable injective Frobenius structures $\tau_{\cF}\col \sig\ast \cF \ra \cF$; see \cite{Ga02} and \cite{Ga03} for details.
If the cokernel of $\tau_\cF$ is supported on the point corresponding to $\wt J$, the restriction of analytic $\tau$-sheaf $(\cF,\tau_\cF)$ to the unit disc $\Sp \wt A_L$ gives rise to an analytic $A(1)$-motive over $L$. 
Consequently, various properties of analytic $\tau$-sheaves are also valid for analytic $A(1)$-motives.
\end{remark}

\begin{definition}
Let $\wt{\ul M}$ be an analytic $A(1)$-motive over $L$.
An analytic $A(1)$-motive $\wt{\ul N}$ is called an {\it  $A(1)$-submotive} of $\ul{\wt M}$ if $\wt N$ is an $\wt A_L$-submodule of $\wt M$ and $\tau_{\wt N}=\tau_{\wt M}|_{\sig\ast \wt N}$.
We say that an $A(1)$-submotive $\ul {\wt N}$ of $\ul{\wt M}$ is {\it saturated} if the map $\tau_{\wt M}{\pmod {\wt N}} \col\sig\ast{\wt M}/{\wt N} \ra {\wt M}/{\wt N}$ induced by $\tau_{\wt M}$ is injective, which means that 
the pair $\ul{\wt M}/\ul{\wt N}=({\wt M}/{\wt N}, \tau_{\wt M}{\pmod {\wt N}})$ is an analytic $A(1)$-motive over $L$.

\end{definition}

\begin{lemma}\label{lem.dimad}
For an exact sequence 
$
0\ra \ul{\wt M'} \ra \ul{\wt M} \ra \ul{\wt M''} \ra 0
$
of analytic $A(1)$-motives, we have $\dim \ul{\wt M}=\dim \ul{\wt M'}+\dim\ul{\wt M''}$. 
If $\ul{\wt M}$ is of height $\leq h$, then $\ul{\wt M'}$ and $\ul{\wt M''}$ are so.
\end{lemma}

\begin{proof}
The snake lemma yields an exact sequence of $L$-vector spaces
\[
0 \lra \Cok(\tau_{\wt M'})\lra \Cok(\tau_{\wt M})  \lra \Cok(\tau_{\wt M''}) \lra 0, 
\]
which proves the claim.
\end{proof}

For an effective $A$-motive $\ul M=(M,\tau_M)$ over $L$, we define the {\it analytification} of $\ul M$ by 
\[
\ul M\ot \wt A_L:=(M\ot_{A_L}\wt A_L,\tau_M\ot \id).
\]
We immediately obtain the following:
\begin{proposition}
The correspondence $\ul M \mapsto \ul M\ot \wt A_L$ yields a functor from the category of effective $A$-motives over $L$
to that of analytic $A(1)$-motives over $L$.
If $\ul M$ is of rank $r$, dimension $d$, and height $\leq h$, then so is $\ul M\ot \wt A_L$.
\end{proposition}

Let $\ul{\wt M}=(\wt M,\tau_{\wt M})$ be an analytic $A(1)$-motive over $L$ of rank $r$.
With $\ul{\wt M}$, one can associate Galois representations by the same way as effective $A$-motives.
For a maximal ideal $\frq$ of $A$,
consider the $\frq$-adic completion $A_{\frq,L^\sep}=A_\frq\wh \ot_{\bF_q}L^\sep$ of $A_{L^\sep}$
and define the $\frq$-adic realization of $\ul{\wt M}$ by 
\[
H^1_\frq(\ul{\wt M},A_\frq):=\{m\in \wt M \ot_{\wt A_L} A_{\frq,L^\sep}\mid \tau_{\wt M}(\sig\ast m)=m\},
\]
which is a free $A_\frq$-module of rank $r$ equipped with a continuous $G_L$-action.
We denote the (rational) $\frq$-adic Tate modules of $\ul{\wt M}$ by 
\[
T_\frq\ul{\wt M}:=\Hom_{A_\frq}(H^1_\frq(\ul{\wt M},A_\frq), A_\frq)\es\es \mbox{and}\es\es V_\frq\ul{\wt M}:=T_\frq\ul{\wt M}\ot_{A_\frq}Q_\frq.
\]

\begin{remark}
If $\ul{\wt M}$ is the analytification of an effective $A$-motive $\ul M$ over $L$, it follows by construction that $T_\frq \ul{\wt M} \cong T_\frq\ul{M}$ as $A_\frq[G_L]$-modules.
\end{remark}

%%%%%%%%%%%%%%%%%%%%%%%%%%%%%%%%%%%%%%%%%%%%%%%%%%%%%%%%%%%%%%%
\subsection{Reduction theory}

%According to the study of analytic $\tau$-sheaves due to Gardeyn in \cite{Ga02} and \cite{Ga03}, we introduce the following notion for analytic $A(1)$-motives.

\begin{definition}
Let $\ul{\wt M}$ be an analytic $A(1)$-motive over $L$.

\begin{itemize}
\item[(1)] A {\it formal model} of $\ul{\wt M}$ is a pair $\ul{\wt\cM}=(\wt \cM,\tau_{\wt \cM})$ consisting of a finite locally free $\wt A_\cO$-module and an injective $\wt A_\cO$-homomorphism $\tau_{\wt\cM}\col \sig\ast \wt\cM \ra \wt\cM$ such that there is an isomorphism $\iota \col \wt M \lrai \wt \cM\ot_{\wt A_\cO}\wt A_L$ with $\tau_{\wt \cM}\circ \sig\ast\iota=\iota \circ \tau_{\wt M}$.
\item[(2)] A formal model $\ul{\wt \cM}$ of $\ul{\wt M}$ is called a {\it good model} if the induced $A_k$-homomorphism 
\[
\tau_{\wt \cM}\ot \id\colon \sig\ast\wt\cM\ot_{\wt A_\cO}A_k\lra  \wt\cM\ot_{\wt A_\cO}A_k
\]
is injective.
In this case, we say that $\ul{\wt M}$ has {\it good reduction}.
\item[(3)] We say that $\ul{\wt M}$ has {\it strongly semi-stable reduction} if there is a filtration (called a {\it semi-stable filtration}) of $\ul{\wt M}$ by saturated $A(1)$-submotives 
\[
0=\ul{\wt M}_0\subset \ul{\wt M}_1 \subset \cdots \subset \ul{\wt M}_n=\ul{\wt M}
\]
such that all the quotients $\ul{\wt M}_i/\ul{\wt M}_{i-1}$ have good reduction.
\end{itemize}
We say that an effective $A$-motive $\ul M$ over $K$ has {\it strongly semi-stable reduction} at a finite place $v$ of $K$ if its analytification $\ul M \ot_{A_K}{\wt A_{K_v}}$ does so. 
\end{definition}

\begin{remark}
By \cite[Theorem 4.7]{HH16}, it follows that $\ul{\wt \cM}$ is a good model of $\ul{\wt M}$ if and only if $\Cok(\tau_{\wt \cM})$ is a finite free $\cO$-module annihilated by some power of $\wt J_\cO$. 
This means that the reduction $(\wt\cM\ot_{\wt A_\cO}A_k,\tau_{\wt \cM}\ot \id)$ becomes an effective $A$-motive over the residue field $k$ having the same rank and dimension of $\ul{\wt M}$.
In addition, if $\ul{\wt M}$ is the analytification of an effective $A$-motive $\ul M$ over $L$, then we see that $\ul{\wt M}$ has good reduction if and if 
$\ul M$ has good reduction, and that the reduction $(\wt\cM\ot_{\wt A_\cO}A_k,\tau_{\wt \cM}\ot\id)$ is isomorphic to $(\cM\ot_{A_\cO}A_k,\tau_{\cM}\ot\id)$ for a good model $(\cM,\tau_{\cM})$ of $\ul M$.
\end{remark}

\begin{remark}
Gardeyn \cite[Definition 4.6]{Ga03} provides a more general notion of semi-stability as follows. 
An analytic $\tau$-sheaf $\ul{\cF}$ on the rigid analytic space $(\Spec A_L)^\an$ is said to be {\it semi-stable} if there is a non-empty open subscheme $\cX$ of $\Spec A_L$ such that the restriction $\cF|_{\cX^\an}$ to the associated rigid analytic space $\cX^\an$ of $\cX$ admits a suitable ``semi-stable filtration''
\[
0=\ul\cF_0 \subset \ul\cF_1 \subset \cdots \subset \ul\cF_n=\cF|_{\cX^\an}.
\] 
If $\cX$ can be taken as $\cX=\Spec A_L$, then $\ul\cF$ is said to be {\it strongly semi-stable}.
Although any analytic $\tau$-sheaf always becomes semi-stable after extending $L$ to a finite extension, there is an example of an analytic $\tau$-sheaf which is never strongly semi-stable even if one extends $L$ to any extensions; see \cite[\S\S 5.3]{Ga03}.  
\end{remark}

By the result \cite[Theorem 5.7]{Ga03} on analytic $\tau$-sheaves due to Gardeyn, we obtain an analytic $A(1)$-motive analog of the N\'eron-Ogg-Shafarevich criterion.
Recall that now $L$ is a completion of $K$ at a place lying above  $\frp$.

\begin{theorem}[Gardeyn]
For an analytic $A(1)$-motive $\ul{\wt M}$ over $L$, the following statements are equivalent.
\begin{itemize}
\item[$(1)$] $\ul{\wt M}$ has good reduction.
\item[$(2)$] For any maximal ideal $\frq$ of $A$ with $\frq\neq \frp$, the  $\frq$-adic representation $V_\frq \ul{\wt M}$ is unramified.
\item[$(3)$] There is a maximal ideal $\frq \neq \frp$ such that $V_\frq\ul{\wt M}$ is unramified. 
\end{itemize}
\end{theorem}

\begin{corollary}\label{cor.unipotent}
If an analytic $A(1)$-motive $\ul{\wt M}$ over $L$ has strongly semi-stable reduction, then for any maximal ideal $\frq$ of $A$ with $\frq \neq \frp$, the  $I_L$-action on $V_\frq\ul{\wt M}$ is unipotent.  

\end{corollary}

\begin{proof}
By the semi-stable filtration of $\ul{\wt M}$, we see that $V_\frq \ul{\wt M}$ is isomorphic up to semi-simplification to a direct sum of $\frp$-adic representations coming from analytic $A(1)$-motives with good reduction.   
Thus the semi-simplification $(V_\frq\ul{\wt M})^\ss$ is unramified.
\end{proof}

For abelian varieties over a field $F$, there is a known criterion for good reduction other than that of N\'eron-Ogg-Shafarevich; 
Namely,  a good reduction criterion from the viewpoint of Barsotti-Tate groups is known;
see \cite[Proposition IX.\ 5.13]{SGA7} (for $\mathrm{ch}(F)=0$) and \cite[2.5]{dJ98} (for $\mathrm{ch}(F)>0$).  
In an equal-characteristic setting, effective local shtukas play the role of $F$-crystals of Barsotti-Tate groups.
As this analogy, Hartl and  H\"{u}sken \cite{HH16} provide the following criterion:

\begin{theorem}[Hartl and  H\"{u}sken]\label{thm.analyticgood}
Let $\ul{\wt M}$ be an analytic $A(1)$-motive over $L$ of height $\leq h$.
Then the following assertions are equivalent.
\begin{itemize}
\item[$(1)$] $\ul{\wt M}$ has good reduction.
\item[$(2)$] There is a $\tau$-equivariant isomorphism 
$f\col \ul{\wt M}\ot_{\wt A_L}A_{\frp,L} \lrai (N,\tau_N)\ot_{A_{\frp,\cO}}A_{\frp,L}$, where $N$ is a
 finite locally free $A_{\frp,\cO}$-module  equipped with an injective $A_{\frp,\cO}$-homomorphism $\tau_N\col \sig\ast N\ra N$ whose cokernel is annihilated by $\wt J^h$.  
\end{itemize}
\end{theorem}
Especially, for   a good model $\ul{\wt \cM}$ of $\ul{\wt M}$, the above pair $(N,\tau_N)$ can be  taken as $(N,\tau_N)=(\cM\ot_{\wt A_\cO}A_{\frp,\cO}, \tau_{\cM}\ot\id)$.  
As we have observed in Remark \ref{rem.localshtukaeq}, such $(N,\tau_N)$ corresponds to an effective local shtuka $\ul{\wh M}$ over $\cO$ of height $\leq h$.
In particular, we obtain the following:

\begin{proposition}\label{prop.localan}
Let $\ul{\wt M}$ be an analytic $A(1)$-motive over $L$ of height $\leq h$.
If $\ul{\wt M}$ has  good reduction, then there is an    effective local shtuka $\ul{\wh M}$ of height $\leq h$ over $\cO$ which has the same rank and dimension of $\ul{\wt M}$, and satisfies $T_\frp\ul{\wt M} \cong T_\frp\ul{\wh M}$ as $A_\frp[G_L]$-modules.
\end{proposition}

\begin{proof}
Let $\ul{\wt \cM}=(\wt \cM,\tau_{\wt \cM})$ be a good model  of $\ul{\wt M}$. 
%and so we obtain a $\tau$-equivariant isomorphism $\wt M \rai \wt\cM \ot_{\wt A_\cO}\wt A_L$.
As  with Example \ref{ex.motiveshtuka}, 
for the ideals $\fra_i$ of $A_{\frp,\cO}$ generated by $\{b\ot 1-1\ot \gamma(b)^{q^i} \mid b\in \bF_\frp\}$,
we consider the decomposition 
$A_{\frp,\cO}=\prod_{i=0}^{d_\frp-1}A_{\frp,\cO}/\fra_i$ and 
define $\ul{\wh M}:=(\wh M,\tau_{\wh M}):=(\wt \cM\ot_{\wt A_\cO}A_{\frp,\cO}/\fra_0,\tau_{\wt \cM}^{d_\frp}\ot\id)$, which
 is an effective local shtuka over $\cO$ of height $\leq h$.
It has the same rank of $\ul{\wt M}$ by construction and  admits an isomorphism
\[
\bigoplus_{i=0}^{d_\frp-1}(\tau_{\wt\cM}\otimes\id)^i\ {\rm mod}\  \fra_i \colon \left( 
\bigoplus_{i=0}^{d_\frp-1}\sig^{i*}\wh M,
\tau_{\wh M}\oplus\bigoplus_{i \neq 0}\id
\right)
\lrai
\ul{\wt \cM}\otimes_{A_{\cO}}A_{\frp,\cO}.
\]
From this, it follows that $\Cok(\tau_{\wh M})\cong \Cok(\tau_{\wt\cM})$ and hence 
$\dim\ul{\wh M}=\dim\ul{\wt M}$.
Now, using the decomposition $A_{\frp,\cO}=\prod_{i=0}^{d_\frp-1}A_{\frp,\cO}/\fra_i$, we have
\[
\wt M\ot_{\wt A_L}A_{\frp,L^\sep} \cong (\wt \cM \otimes_{\wt A_\cO} A_{\frp,\cO})\ot_{A_{\frp,\cO}}A_{\frp,L^\sep}\cong \prod_{i=0}^{d_\frp-1}\wt \cM \ot_{\wt A_\cO}A_{\frp,L^\sep}/\fra_i.
\]
Thus the element $m=(m_i)\in \wt M\ot_{\wt A_L}A_{\frp,L^\sep}=\prod_{i=0}^{d_\frp-1}\wt \cM \ot_{\wt A_\cO}A_{\frp,L^\sep}/\fra_i$
satisfies $\tau_{\wt M}(\sig\ast m)=m$ if and only if $m_{i+1}=\tau_{\wt \cM}(\sig\ast m_i)$ for $i \neq d_\frp-1$ and $m_0=\tau_{\wt \cM}(\sig\ast m_{d_\frp-1})$, so that $m_0=\tau_{\wt\cM}^{d_\frp}(\sig^{d_\frp *}m_0)=\tau_{\wh M}(\hsig\ast m_0)$
 and $m=(m_i)$ can be recovered by $m_0$.
Hence $(m_i)\mapsto m_0$ gives rise to the isomorphism $H^1_\frp(\wt {\ul M},A_\frp) \rai \wc T_\frp\ul{\wh M}$ and so we get $T_\frp\ul{\wt M} \cong T_\frp\ul{\wh M}$ by taking dual of both sides.
\end{proof}

\begin{theorem}\label{thm.tamemot}
Let $\ul M$ be an effective $A$-motive over $L$ of height $\leq h$ and suppose it has strongly semi-stable reduction.
If  the ramification index $e$ of the extension $L/\gamma(Q_\frp)$ satisfies $eh< q_\frp-1$, then $\TI(V_\frp\ul M) \subset [0,eh]$ and the equation 
\[
e\cdot \dim\ul M=\Sigma \TI(V_\frp \ul M)
\] 
holds.
\end{theorem}

\begin{proof}
By assumption, the analytification $\ul{\wt M}:=\ul{M}\ot \wt A_L$ of $\ul M$ admits a semi-stable filtration
\[
0=\ul{\wt M}_0\subset \ul{\wt M}_1 \subset \cdots \subset \ul{\wt M}_n=\ul{\wt M}
\]
and so we have $V_\frp\ul M \simeq_\ss \bigoplus_{i=1}^nV_\frp(\Gr_i\ul {\wt M})$, where $\Gr_i\ul {\wt M}:=\ul{\wt M}_i/\ul{\wt M}_{i-1}$ for each $i$.
This implies that $\TI(V_\frp\ul M)=\coprod_{i=1}^n \TI \left( V_\frp(\Gr_i\ul {\wt M}) \right)$.
By Lemma \ref{lem.dimad},  each $\Gr_i\ul {\wt M}$ 
is of height $\leq h$ and $\dim \ul M=\sum_{i=1}^n \dim \Gr_i\ul {\wt M}$.
Since all $\Gr_i\ul {\wt M}$ have good reduction, Proposition \ref{prop.localan} implies that  there are effective local shtukas $\ul{\wh M}_i$ of height $\leq h$ such that  
$\dim \Gr_i\ul{\wt M}=\dim \ul{\wh M}_i$ and $V_\frp\left(\Gr_i\ul{\wt M}\right) \cong V_\frp\ul{\wh M}_i$.
Applying Proposition \ref{prop.localtame} to each $\ul{\wh M}_i$, we get the conclusion.
\end{proof}

%%%%%%%%%%%%%%%%%%%%%%%%%%%%%%%%%%%%%%%%%%%%%%%%%%%%%%%%%%%%%%%%%%%%%%%%%%%%%%%%%%%%%
%
%   6. Congruences of Galois representations 
%
%%%%%%%%%%%%%%%%%%%%%%%%%%%%%%%%%%%%%%%%%%%%%%%%%%%%%%%%%%%%%%%%%%%%%%%%%%%%%%%%%%%%%%

\section{Congruences of Galois representations}\label{sec.result}

In what follows, we consider effective $A$-motives over the $A$-field $(K,\gamma\col A\ra K)$  as in \S\ref{section.intro} and associated $\frp$-adic representations of $G_K$.
The goal of this section is to prove our results on two $\frp$-adic representations of $G_K$ arising from effective $A$-motives with strongly semi-stable reduction at finite places $v,u$ of $K$, where often $v$ is not lying above $\frp$ and $u \mid \frp$.
The key ingredient is the {\it Weil weights} at $v$ of strongly semi-stable effective $A$-motives, which are determined by eigenvalues of a lift $\Frob_v \in G_v$ of the arithmetic Frobenius $\phi_v \in G_{\bF_v}$.  
Recall that we denote by $d_v=[\bF_v:\bF_q]$ and $q_v=\#\bF_v=q^{d_v}$, and that we write $V_v=V|_{G_v}$ for a representation $V$ of $G_K$.
As in \S\S \ref{subsec.notation}, we fix an algebraic closure $\bar Q$ of $Q$ and an extension $|\cdot|_\infty$ of the normalized absolute value at the place $\infty$ of $Q$.

\subsection{Weil weights}

Let $v$ be a finite place of $K$.

\begin{lemma}\label{lem.charpoly}
Let $\ul{\bar{M}}=({\bar{M}}, \tau_{\bar{M}})$ be an effective $A$-motive over $\bF_v$ and $\frp \subset A$ a maximal ideal with $v\nmid \frp$.
Then  the characteristic polynomial $\det(X-\phi_v\mid T_\frp\ul{\bar{M}})$ has coefficients in $A$ which are independent of $\frp$.
\end{lemma}

\begin{proof}
It is known  that
there is a group homomorphism $\tau\col \bar M\ra \bar M$ 
satisfying $\tau(\lam m)=\sig(\lam)\tau(m)$ for $\lam \in A_{\bF_v}$ such that $\tau_{\bar{M}}(\sig\ast m)=\tau(m)$ for any  $m\in \bar M$.
Since $A_{\bF_v}$ is finite flat over $A$, it follows that $\bar M$ is finite projective over $A$.
Thus there is an $A$-module $P$ such that $\bar{M} \op P$ is finite free over $A$.
Let $x$ be an indeterminate and consider the zero map $0\colon P \ra P$.
Then the characteristic polynomial of $\tau\op 0$ on the free $A$-module $\bar M\op P$
\[
\det(\id-x\tau \mid \bar{M}):=\det(\id-x(\tau\op 0) \mid \bar{M}\op P) \in A[x]
\]
is independent of the choice of $P$; see \cite[Lemma-Definition 8.1.1]{BP09}.
Now the $d_v$-th iteration $\sig^{d_v}=\sig \circ \cdots \circ \sig$ is identity on $\bF_v$ and hence  ${\tau}^{d_v}\colon \bar{M}
\ra \bar{M}$ is $A_{\bF_v}$-linear.
By \cite[Lemma 8.1.4]{BP09}, we see that 
$
\det(\id-x\tau \mid \bar{M})=\det(\id-x^{d_v}\tau^{d_v} \mid \bar{M})
$, where $\bar M$ is regarded as an $A_{\bF_v}$-module in the right hand side. 
In particular, we have $\det(\id-x\tau \mid \bar{M}) \in 1+x^{d_v}A[x^{d_v}]$.
Here by \cite[\S6]{TW96}, we have
$
\det(\id-x^{d_v}\phi_v \mid T_\frp \ul{\bar{M}})=\det(\id-x^{d_v}\tau^{d_v} \mid \bar{M}).
$
Setting $X=x^{-d_v}$ and $r=\rk \ul{\bar M}$, we get
$
\det(X-\phi_v\mid T_\frp\ul{\bar{M}})=X^{r}\cdot\det(\id-x\tau \mid \bar{M})  \in A[X].
$

\end{proof}

\begin{proposition}\label{prop.charpolystrong}
Let $\ul M$  be an effective $A$-motive over $K$ with strongly semi-stable reduction at $v$ and $\frp$ is  a maximal ideal of $A$ with $v \nmid \frp$.
Then for $V=V_\frp\ul M$, 
 the characteristic polynomial 
\[
P_{\ul M,v}(X):=\det(X-\Frob_v \mid V_v^\ss)
\] has coefficients in $A$ which are independent of $\frp$.
\end{proposition}

\begin{proof}
Let $\ul{\wt M}:=\ul M\ot \wt A_{K_v}$ be the analytification of $\ul M$ and consider a semi-stable filtration 
\[
0=\ul{\wt M}_0\subset \ul{\wt M}_1 \subset \cdots \subset \ul{\wt M}_n=\ul{\wt M}
\]
of $\ul{\wt M}$.
Set $\Gr_i\ul{\wt M}:=\ul{\wt M}_i/\ul{\wt M}_{i-1}$ for each $1 \leq i \leq n$.
Since $V_v \simeq_\ss \bigoplus_{i=1}^n V_\frp(\Gr_i\ul{\wt M})$ and each $V_\frp(\Gr_i\ul{\wt M})$ is unramified at $v$, we obtain 
 \[
 P_{\ul M,v}(X)=\prod_{i=1}^n\det(X-\Frob_v\mid V_\frp(\Gr_i\ul{\wt M}))
 =\prod_{i=1}^n \det(X-\Frob_v \mid T_\frp(\Gr_i\ul{\wt M})).
 \]
Now let $\ul{\wt \cM}_i=(\wt \cM_i,\tau_i)$ be a good model of $\Gr_i\ul{\wt M}$ for each $i$.
Then the reductions $\ul{{\wt \cM}}_i\ot A_{\bF_v}=(\wt\cM_i\ot_{\wt A_{\cO_v}}A_{\bF_v},\tau_i\ot\id)$ are effective $A$-motives over $\bF_v$.
Since $T_\frp(\Gr_i\ul{\wt M})$ is unramified at $v$, it is isomorphic to  
$T_\frp (\ul{{\wt \cM}}_i\ot A_{\bF_v})$ as  representations of $G_v/I_v \cong G_{\bF_v}$.
Hence $\det(X-\Frob_v \mid T_\frp(\Gr_i\ul{\wt M}))=\det(X-\varphi_v\mid T_\frp (\ul{{\wt \cM}}_i\ot A_{\bF_v}))$ holds for any $i$ and so the claim is proved by Lemma \ref{lem.charpoly}. 
\end{proof}

\begin{definition}
Suppose that $\ul M$ has strongly semi-stable reduction at $v$.
Let $\alpha_1,\ldots,\alpha_{\rk \ul M} \in \bar Q$ be the roots of $P_{\ul M,v}(X)$ and define the rational numbers $w_i \in \bQ$ by
$|\alpha_i|_\infty=q_v^{w_i}$. 
We call $w_1,\ldots,w_{\rk \ul M}$  the {\it Weil weights of $\ul M$ at v}, and denote by $\W_v(\ul M):=\{w_1,\ldots,w_{\rk \ul M}\}$ the multi-set consisting of them.
Note that $\W_v(\ul M)$ is independent of the choice of extensions of $|\cdot|_\infty$ to $\bar Q$.
\end{definition}

\begin{example}
Let $\ul M$ be an effective pure $A$-motive over $K$ of rank $r$ and dimension $d$.
If $\ul M$ has good reduction at $v \nmid \frp$, then $P_{\ul M,v}(X)=\det(X-\Frob_v\mid T_\frp\ul{M})$ and so  Proposition \ref{prop.pureweight} implies
\[
\W_v(\ul M)=\left\{\frac{d}{r}, \ldots, \frac{d}{r}\right\},
\]
where the multiplicity of $\frac{d}{r}$ is $r$.
Then we particularly obtain $\Sigma\W_v(\ul M)=d$.
\end{example}

We can see that the equation  $\Sigma\W_v(\ul M)=d$ also holds in the  strongly semi-stable case:

\begin{proposition}\label{prop.freigen}
Let $\ul M$ be an effective $A$-motive over $K$ of rank $r$.
Suppose that $\ul M$ has strongly semi-stable reduction at a finite place $v$ of $K$.
We set $P_{\ul M,v}(X)=\sum_{i=0}^{r}a_iX^{r-i} \in A[X]$.
Then the following assertions hold.
\begin{itemize}
\item[$(1)$] $\Sigma\W_v(\ul M)=\dim \ul M$.
\item[$(2)$] Suppose that all Weil weights of $\ul M$ at $v$ are non-negative.
Then $|a_i|_\infty \leq q_v^{i\dim \ul M}$ for any $0 \leq i \leq r$. 
\end{itemize}
\end{proposition}

\begin{proof}
(1) Let $\frp$ be a maximal ideal of $A$ with $v\nmid \frp$.
As we see in Proposition \ref{prop.rkdim},  the determinant $\det \ul M$ of $\ul M$ is an effective $A$-motive over $K$ with $\rk \det \ul M=1$
and $\dim \det\ul M=\dim \ul M$.
Then $\det(V_\frp\ul M) \cong V_\frp(\det \ul M)$ as 
$\frp$-adic representations of $G_K$.
Since the $I_v$-action on $V_\frp \ul M$ is unipotent by Corollary \ref{cor.unipotent}, $V_\frp(\det \ul M)$ is unramified at $v$ and hence $\det \ul M$ has good reduction at $v$.
From this, we have $\W_v(\det \ul M)=\{\Sigma \W_v(\ul M)\}$ because the $\Frob_v$-eigenvalue of $V_\frp(\det \ul M)$ is the product of all roots of $P_{\ul M,v}(X)$.
Since   $\det \ul M$ is pure of weight equal to $\dim \ul M$ by Lemma \ref{lem.purerkone}, we get  $\dim \ul M=\Sigma \W_v(\ul M)$.

(2) Let $\alpha_1,\ldots, \alpha_r$ be  the roots of $P_{\ul M,v}(X)=\sum_{i=0}^ra_iX^{r-i}$.
Then obviously 
$a_0=1$.
For any $1 \leq i \leq r$, we have 
\[
a_i= \sum_{1 \leq s_1<\cdots < s_i \leq r}(-\alpha_{s_1})(-\alpha_{s_2})\cdots(-\alpha_{s_i}).
\]
Since all $\alpha_i$ satisfy $|\alpha_i|_\infty \geq 1$ by assumption, we have
$
|\alpha_i|_\infty \leq |a_r|_\infty=q_v^{\dim \ul M}.
$
Therefore we have
\[
|a_i|_\infty \leq \max_{1 \leq s_1<\cdots < s_i \leq r}\{|\alpha_{s_1}\cdots\alpha_{s_i}|_\infty\}\leq q_v^{i\dim \ul M}.
\]
\end{proof}

\subsection{Results}\label{ss.results}

To state our results, we recall the notation introduced in \S\S \ref{subsec.motivation}.
Let $r>0$ be a positive integer as usual.
The injective $A$-field structure $\gamma\col A\ra K$ extends to the field embedding $\gamma\col Q \ra K$ and then $K/\gamma(Q)$ is a finite extension. 
Let $K_{\rm s}$ denote the maximal separable extension of $\gamma(Q)$ in $K$, so that $K/K_{\rm s}$ is purely inseparable if $K\neq K_{\rm s}$.
Then we write $[K:Q]_{\rm i}:=[K:K_{\rm s}]$ for the inseparable degree of $K/\gamma(Q)$.
%\begin{align*}
%[K:Q]&:=[K:\gamma(Q)], \\
%[K:Q]_{\rm s}&:=[K_{\rm s}:\gamma(Q)], \  \mbox{and}\\
%[K:Q]_{\rm i}&:=[K:K_{\rm s}]
%\end{align*}
%for $K/K_{\rm s}/\gamma(Q)$.
Denote by $\frd=\frd_{K_{\rm s}/\gamma(Q)}$ the relative discriminant of $K_{\rm s}/\gamma(Q)$.
Then we define 
\[
D_K=
\begin{cases}
\max\{d_{\frq} \mid \frq\subset A\ \mbox{is a maximal ideal dividing}\  \frd  \} & \mbox{if}\ \frd \neq A,\\
1 & \mbox{if}\ \frd=A.
\end{cases}
\]
Now it is known that purely inseparable extensions of global function fields have to be totally ramified at all places by \cite[Proposition 7.5]{Ros02}. Hence, for any finite place $u$ of $K$ lying above a maximal ideal $\frp$ of $A$ with $d_\frp>D_K$, the ramification index $e_{u\mid\frp}$ of $u\mid \frp$ is equal to $[K: Q]_{\rm i}$.

\begin{definition}\label{defi.mot}
For a non-negative $h \geq 0$ and two finite places $v, u$ of $K$, denote by 
\[
\mathsf{Mot}_{K,r,v}(u,h)
\]
 the set of effective $A$-motives $\ul M$ over $K$ of rank $r$ which satisfy the following conditions:
\begin{itemize}
\item $\ul M$ has strongly semi-stable reduction at both $v$ and $u$,
\item All Weil weights of $\ul M$ at $v$ are non-negative,
\item $\ul M$ is of height $\leq h$.
\end{itemize}
\end{definition}

Our first result is the following criterion for congruent $\frp$-adic representations arising from strongly semi-stable effective $A$-motives, which is a function field analog of the results on  $\ell$-adic representations due to Ozeki and Taguchi \cite{OT14}:

\begin{theorem}\label{thm.main1}
Let $r$ and $h$ be as above and fix a finite place $v$ of $K$.
For any maximal ideal $\frp$ of $A$ with $d_\frp>\max\{d_vr^2h, [K:Q]_{\rm i}h, D_K\}$
and any finite place $u$ of $K$ lying above $\frp$, the following holds: 
For any $\ul M \in \mathsf{Mot}_{K,r,v}(u,h)$ and $\ul M' \in \mathsf{Mot}_{K,r,v}\left(u,(q_\frp-2) [K:Q]_{\rm i}\inv\right)$, if both
\[
\begin{cases}
V_\frp\ul M|_{G_v} \equiv_\ss V_\frp\ul M'|_{G_v} \pmod \frp, & \mbox{and}\\
V_\frp\ul M|_{G_u} \equiv_\ss V_\frp\ul M'|_{G_u} \pmod \frp & 
\end{cases}
\]
hold, 
 then one has 
$V_\frp\ul M|_{G_v} \simeq_\ss V_\frp\ul M'|_{G_v}$, $\dim \ul M=\dim \ul M'$, and $\cW_v(\ul M)=\cW_v(\ul M')$.
\end{theorem}

\begin{proof}
Let us first see what we immediately get from the assumption $d_\frp>\max\{d_vr^2h, [K:Q]_{\rm i}h, D_K\}$.
Since  particularly $d_v<d_\frp$, we see that $v$ is not lying above $\frp$.
For the maximum of Hodge-Pink weights $h'=\max \HP(\ul M')$ of $\ul M'$, it follows that $\ul M'$ is of height $\leq h'$ and $h' \leq (q_\frp-2)[K:Q]_{\rm i}\inv$.
As mentioned previously, we have $e_{u\mid \frp}=[K:Q]_{\rm i}$ by $d_\frp>D_K$.
Therefore we have $e_{u\mid\frp}h'<q_\frp-1$.
On the other hand, we have $e_{u\mid\frp}h< q_\frp-1$ as
\[
e_{u\mid\frp}h=[K:Q]_{\rm i}h<d_\frp<q^{d_\frp}-1=q_\frp-1.
\]
Note that the inequality $d_\frp<q^{d_\frp}-1$ follows as $d_\frp>D_K\geq 1$.

Now write $V:=V_\frp\ul M$ and $V':=V_\frp\ul M'$ for the associated $\frp$-adic representations of $G_K$.
Then $V_v \simeq_\ss V'_v$ holds if and only if $P_{\ul M,v}(X)=P_{\ul M',v}(X)$ by Proposition \ref{prop.chp}, and $\W_v(\ul M)=\W_v(\ul M')$ in this case. 
Let $P_{\ul M,v}(X)=\sum_{i=0}^ra_iX^{r-i}$ and $P_{\ul M',v}(X)=\sum_{i=0}^ra_i'X^{r-i}$, and set $c_i:=a_i-a'_i \in A$.
To see $P_{\ul M,v}(X)=P_{\ul M',v}(X)$, 
it suffices to check that $c_i=0$ for any $i$.
Since  $V_v \equiv_\ss V_v' \pmod \frp$, we obtain 
$P_{\ul M,v}(X)\equiv P_{\ul M',v}(X) \pmod\frp$.
Hence $c_i \equiv 0 \pmod\frp$ for any $i$. 
Applying  Proposition \ref{prop.freigen} (2) to $P_{\ul M,v}(X)$ and $P_{\ul M',v}(X)$, we get the estimate 
\[
|c_i|_\infty \leq \max\{|a_i|_\infty, |a'_i|_\infty\} \leq q_v^{i\max\{\dim \ul M,\dim\ul M'\}} \leq q^{d_vr\max\{\dim \ul M,\dim\ul M'\}}
\]
 for any $0 \leq i \leq r$.
 On the other hand, the assumption $V_u \equiv_\ss V'_u \pmod\frp$ implies 
 $\TI(V_u)=\TI(V_u')$ as multi-sets. 
Since we have  $e_{u\mid\frp}h < q_\frp-1$ and $e_{u\mid\frp}h' < q_\frp-1$, Theorem \ref{thm.tamemot} provides the equation 
\[
e_{u\mid\frp}\cdot\dim \ul M=\Sigma \TI(V_u)=\Sigma \TI(V'_u)=e_{u\mid\frp}\cdot\dim \ul M'.
\]
Thus $\ul M$ and $\ul M'$ have the same dimension.
Since $\ul M$ is of height $\leq h$ and the multi-set $\HP(\ul M)$ consists of $r$ integers, it follows by Proposition \ref{prop.HPwt} that 
$\dim \ul M=\Sigma\HP(\ul M) \leq rh$.
Consequently, we obtain 
\[
|c_i|_\infty \leq q^{d_vr^2h}<q^{d_\frp}
\]
and so Lemma \ref{lem.absvalue} implies $c_i=0$ for any $i$
.
Thus we get $V_v \simeq_\ss V'_v$ and $\W_v(\ul M)=\W_v(\ul M')$.
\end{proof}

For a maximal ideal $\frp$ of $A$, we consider the union 
$
\mathsf{Mot}_{K,r,v}(\frp,h):=\coprod_{u\mid \frp} \mathsf{Mot}_{K,r,v}(u,h),
$
where $u$ runs through all places of $K$ lying above $\frp$.
As a corollary of Theorem \ref{thm.main1}, we have:

\begin{corollary}
Let $r,h,$ and $v$ be as above.
Fix an effective $A$-motive $\ul M'$ over $K$ of rank $r$ such that $\ul M'$
has strongly semi-stable reduction at $v$ and all Weil weights at $v$ are non-negative.
Then there is a constant $C=C(K,r,h,v, \ul M')>0$ such that if a maximal ideal $\frp$ of $A$ satisfies $d_\frp>C$, then there exist no effective $A$-motives 
$\ul M \in \mathsf{Mot}_{K,r,v}(\frp, h)$ satisfying $\W(\ul M) \neq \W(\ul M')$ and $V_\frp\ul M \equiv_\ss V_\frp\ul M' \pmod \frp$ as representations of $G_K$.
\end{corollary}

\begin{proof}
Since $\ul M'$ has good reduction at all but finitely many places of $K$, we may choose a constant $C >\max\{d_vr^2h, [K:Q]_{\rm i}h, D_K\}$ so large with respect to $\ul M'$
 such  that if $d_\frp>C$, then $\ul M'$ is of height $\leq (q_\frp-2)\cdot[K:Q]_{\rm i}\inv$ and has good reduction at all places lying above $\frp$.
Hence if $\ul M \in \mathsf{Mot}_{K,r,v}(\frp,h)$ satisfies $V_\frp\ul M \equiv_\ss V_\frp\ul M' \pmod \frp$, then 
$\W_v(\ul M)=\W_v(\ul M')$ must hold by
Theorem \ref{thm.main1}.
\end{proof}

Our second result is motivated by the non-existence conjecture on abelian varieties with constrained torsion due to Rasmussen and Tamagawa \cite{RT08}: 

\begin{theorem}\label{thm.main2}
Let $r \geq 2$ be an integer and $h\geq 0$.
Let $v$ be a finite place of $K$.
For any maximal ideal $\frp$  of $A$ with $d_\frp>\max\{d_vr^2h[K:Q]_{\rm i}, D_K\}$ and any finite place $u$ of $K$ lying above $\frp$, there 
exist no effective $A$-motives $\ul M \in \mathsf{Mot}_{K,r,v}(u,h)$ satisfying the following conditions:
\begin{itemize}
\item There exist  at least one Weil weight $w \in \W_v(\ul M)$ such   that $[K:Q]_{\rm i}w$ is a non-integer,
\item There exist one-dimensional rank-one effective $A$-motives $\ul M_1,\ldots,\ul M_r$ over $K$ having good reduction at both
$v$ and $u$ such that 
%\[
%V_\frp\ul M \equiv_\ss \bigoplus_{i=1}^rV_\frp(
%\ul{M}_i^{\ot m_i}) \pmod \frp
%\]
%holds as $G_K$-representations.
\[
\begin{cases}
\DS V_\frp\ul M|_{G_v} \equiv_\ss \bigoplus_{i=1}^rV_\frp(
\ul{M}_i^{\ot m_i})|_{G_v} \pmod \frp
, & \mbox{and}\\
\DS V_\frp\ul M|_{G_u} \equiv_\ss \bigoplus_{i=1}^rV_\frp(
\ul{M}_i^{\ot m_i})|_{G_u} \pmod \frp &
\end{cases}
\]
hold for some non-negative integers $m_1,\ldots,m_r$.
\end{itemize}
\end{theorem}

\begin{remark}
If $r=1$, then $\W_v(\ul M)=\{\dim \ul M\}$ and so the first condition of the theorem never holds.
Hence it suffices to consider only the case $r \geq 2$. 
\end{remark}

\begin{proof}
Set $V:=V_\frp\ul M$ and $V':=\bigoplus_{i=1}^rV_\frp(
\ul{M}_i^{\ot m_i})$, so that $V'$ is semi-simple.
Let $\chi_i:G_K\ra \bF_\frp\mal$ be the character determined by $T_\frp\ul M_i\ot_{A_\frp}\bF_\frp$ for each $i$.
Since $\chi_i|_{G_u}$ is coming from a rank-one effective local shtuka of dimension one, it follows by Example \ref{ex.rankone} that $\chi_i|_{I_u}=\theta_1^{e_{u\mid \frp}}$, where $\theta_1\col I_u \ra \bF_\frp\mal$ is the level-one fundamental character. 
Now we have $e_{u\mid\frp}=[K:Q]_{\rm i}$ by  $d_\frp>D_K$ and therefore the $I_u$-action on $T_\frp\ul M_i^{\ot m_i}\ot_{A_\frp} \bF_\frp$ is given by $\theta_1^{[K:Q]_{\rm i}m_i}$.
For each $i$, let $n_i$ be the integer satisfying $\chi_i^{m_i}|_{I_u}=\theta_1^{n_i}$ and  $0 \leq n_i \leq q_\frp-2$.
Then we have $\TI(V'_u)=\{n_1,\cdots,n_r\}$ as multi-sets and  
 $[K:Q]_{\rm i}m_i \equiv n_i \pmod {q_\frp-1}$ for any $i$.
Since $V_u \equiv_\ss V_u' \pmod \frp$ by assumption, we have $\TI(V_u)=\TI(V_u')$.
Applying Theorem \ref{thm.tamemot} to $\ul M$, we  have $\TI(V_u) \subset [0,e_{u\mid\frp}h]$ and so
all $n_i$ satisfy $0 \leq n_i \leq h[K:Q]_{\rm i}$.

Now let us consider the $[K:Q]_{\rm i}$-power $\Frob_v^{[K:Q]_{\rm i}}$ of $\Frob_v$.
Then the characteristic polynomials of it  on $V_v^\ss$ and $V_v'$ have coefficients in $A$.  
We set $P_{\ul M,v}^{\SSC [K:Q]_{\rm i}}(X):=\det(X-\Frob_v^{[K:Q]_{\rm i}} \mid V_v^\ss)$.
Define $\ul M':=\bigoplus_{i=1}^r\ul M_i^{\ot n_i}$ (Recall that $\ul M_i^{\ot n_i}=\ul{A_K}$ if $n_i=0$).
Then  by $V_v \equiv_\ss V_v' \pmod \frp$ and $[K:Q]_{\rm i}m_i \equiv n_i \pmod {q_\frp-1}$, we obtain 
\begin{align*}
P_{\ul M,v}^{\SSC [K:Q]_{\rm i}}(X)=\det(X-\Frob_v^{[K:Q]_{\rm i}} \mid V_v^\ss) & \equiv \det(X-\Frob_v^{[K:Q]_{\rm i}}\mid V_v')  \\
&\equiv \prod_{i=1}^r(X-\chi_i(\Frob_v^{[K:Q]_{\rm i}})^{m_i}) \\
&\equiv \prod_{i=1}^r(X-\chi_i(\Frob_v)^{n_i}) \\
& \equiv P_{\ul M',v}(X) & \pmod \frp.
\end{align*} 
Since each $\ul M_i$ is pure of weight one by Lemma \ref{lem.purerkone}, the multi-set of Weil weights at $v$ of $\ul M'$ is $\W_v(\ul M')=\{n_1,\ldots,n_r\}=\TI(V'_u)$. 
By $0 \leq n_i \leq h[K:Q]_{\rm i}$, we have $\dim \ul M'=\Sigma \W_v(\ul M') \leq rh[K:Q]_{\rm i}$.
If we write $P_{\ul M',v}(X)=\sum_{i=0}^ra_i'X^{r-i} \in A[X]$, then it follows by Proposition \ref{prop.freigen} that 
\[
|a_i'|_\infty \leq q_v^{r\dim \ul M'} \leq q^{d_vr^2h[K:Q]_{\rm i}}
\]
for each $0 \leq i \leq r$.
On the other hand, 
since any root $\beta_i \in \bar Q$ of $P_{\ul M,v}^{\SSC [K:Q]_{\rm i}}(X)$ is a $[K:Q]_{\rm i}$-power of  a root of $P_{\ul M,v}(X)$, we have $|\beta_i|_\infty=q_v^{w_i[K:Q]_{\rm i}}$ for some $w_i \in \W_v(\ul M)$.
Thus if   
we write $P_{\ul M,v}^{\SSC [K:Q]_{\rm i}}(X)=\sum_{i=0}^rb_iX^{r-i} \in A[X]$, then 
\[
|b_i|_\infty \leq |b_r|_\infty=|\beta_1\cdots\beta_r|_\infty =q_v^{r\dim \ul M [K:Q]_{\rm i}} \leq q^{d_vr^2h[K:Q]_{\rm i}}
\] 
for each $0 \leq i \leq r$.
Hence by $d_\frp>d_vr^2h[K:Q]_{\rm i}$, Lemma \ref{lem.absvalue} implies that $P_{\ul M,v}^{\SSC [K:Q]_{\rm i}}(X)=P_{\ul M',v}(X)$.
From this, we see that all $w \in W_v(\ul M)$ satisfy $[K:Q]_iw \in \W_v(\ul M')$.
However, the multi-set $\W_v(\ul M')=\{n_1,\ldots,n_r\}$ contains only integers, which contradicts the assumption.
\end{proof}

Specializing Theorem \ref{thm.main2}  to  effective pure $A$-motives, we immediately obtain the next corollary:

\begin{corollary}
Let $r\geq 2$ be an integer and $\frq$ a maximal ideal of $A$.
Let $d>0$ be a positive integer such that  $d[K:Q]_{\rm i}$ is not divisible by $r$.
If a maximal ideal $\frp$ of $A$ satisfies $d_\frp>\max\{d_\frq r^2d[K:Q]\cdot[K:Q]_{\rm i}, D_K\}$, then there exist no $d$-dimensional effective pure $A$-motives $\ul M$ over $K$ of rank $r$ such that 
\begin{itemize}
\item $\ul M$ has good reduction at a place of  $K$ lying above $\frq$,
\item $\ul M$ has strongly semi-stable reduction at a place of $K$ lying above $\frp$,
\item There exist one-dimensional rank-one effective $A$-motives $\ul M_1,\ldots,\ul M_r$ over $K$ with good reduction outside $\infty$, and non-negative integers $m_1,\ldots,m_r$ such that  
\[
V_\frp\ul M \equiv_\ss \bigoplus_{i=1}^rV_\frp(
\ul{M}_i^{\ot m_i}) \pmod \frp
\]
holds as $\frp$-adic representations of $G_K$.
\end{itemize} 
\end{corollary} 

\begin{proof}
Let $\frp$ be as  $d_\frp>\max\{d_\frq r^2d[K:Q]\cdot[K:Q]_{\rm i}, D_K\}$ and assume that there is an $\ul M$ satisfying the conditions as above.
Let $v \nmid \frp$ and $u\mid \frp$ be finite places of $K$ such that $\ul M$ has good (resp.\ strongly semi-stable) reduction at $v$ (resp.\ at $u$).
Then we have $\ul M \in \mathsf{Mot}_{K,r,v}(u,d)$.
By $d_v=d_{\frq}[\bF_v:\bF_\frq]<d_\frq[K:Q]$, we have $d_vr^2d[K:Q]_{\rm i}<d_\frp$.
Since $\ul M$ is pure and has good reduction at $v$, the Weil weights at $v$ are $\W_v(\ul M)=\{\frac{d}{r}, \ldots, \frac{d}{r}\}$ and  $\frac{d}{r}[K:Q]_{\rm i}$ is a non-integer by assumption.
This contradicts the non-existence in Theorem \ref{thm.main2}.
\end{proof}

%%%%%%%%%%%%%%%%%%%%%%%%%%%%%%%%%%%%%%%%%%%%%
%
% Acknowledgements
%
%%%%%%%%%%%%%%%%%%%%%%%%%%%%%%%%%%
%\section*{Acknowledgments}
%
%This section should come before the References. Dedications and funding
%information may also be included here.

%%%%%%%%%%%%%%%%%%%%%%%%%%%%%%%%%%%%%%%%%%%%%%%%%%%%%%%%%%%
%
% Appendix
%
%%%%%%%%%%%%%%%%%%%%%%%%%%%%%%%%%%%%%%%%%%%%%%%%%%%%%%%%%%%
\appendix

\section{{Proof of Throrem \ref{thm.BN}}} \label{Appendix}

This Appendix based on \cite{Ya} aims to give a proof of  Theorem \ref{thm.BN}.
In fact, Theorem \ref{thm.BN} follows from a general statement (Proposition \ref{prop.BN} below) under the following setting.
From now on, let $\bsA$ be a (not necessarily commutative) ring with unity and consider a (left) $\bsA$-module $M$ .
(What we have in mind as $\bsA$ and $M$ are $\bsA=F[G]$ and $M=V^\ss$.) 
Denote by $\End(M)$ the ring of endomorphisms of $M$ as an abelian group, and by $\End_{\bsA}(M)$ its subring consisting of $\bsA$-endomorphisms.  
Then the {\it bicommutant} $\bsB_M$ of $M$ is defined as
\[
\bsB_M=\{f\in \End(M) \mid f\circ g=g\circ f\es \mbox{for any}\ g\in \End_{\bsA}(M)\}, 
\]
which is a subring of $\End(M)$.
For any $a\in \bsA$, we write $[a]_M$ for the {\it homothey} $[a]_M\colon M\ra M; x\mapsto ax$ and denote by $\bsA_M=\{[a]_M\mid a\in \bsA\}$ the ring of  homotheties.
Then $\bsA_M \subset \bsB_M$. 

\begin{theorem}[Density theorem]\label{thm.density}
Suppose that $M$ is semi-simple.
Then for any $b\in \bsB_M$ and any finite sequence $x_1,\ldots, x_n \in M$,
there exists an element $a\in \bsA$ such that $[a]_M(x_i)=b(x_i)$ for each $1\leq i \leq n$. 
\end{theorem}

\begin{proof}
See \cite[Chapitre 8, \S5, Th\'eor\`eme 3]{Bou81} and its proof.
\end{proof}

\begin{corollary}\label{cor.bicom}
Let $M$ be a semi-simple $\bsA$-module.
If $M$ is finitely generated as a left $\End_{\bsA}(M)$-module, then $\bsA_M=\bsB_M$.
\end{corollary}

\begin{proof}
We have already seen $\bsA_M \subset \bsB_M$.
To prove the converse, we take an element $b \in \bsB_M$ and a generating set $\{x_1,\ldots,x_n\}$ of $M$ as a left $\End_{\bsA}(M)$-module.
By Theorem \ref{thm.density}, there exists $[a]_M \in \bsA_M$ such that $b(x_i)=[a]_M(x_i)$ for each $i$.
Since any $x\in M$ is written by $x=\sum_{i=1}^ng_i(x_i)$ for $g_i\in \End_{\bsA}(M)$, it follows that 
\[
[a]_M(x)=a\sum_{i=1}^ng_i(x_i)=\sum_{i=1}^ng_i(ax_i)=\sum_{i=1}^ng_i(b(x_i))=\sum_{i=1}^nb\left(g_i(x_i)\right)=b(x).
\]
Thus $\bsB_M \subset \bsA_M$.
\end{proof}

\begin{corollary}\label{cor.3}
Let $M_1,\ldots,M_n$ be simple $\bsA$-modules which are not isomorphic to each other.
Suppose that each $M_i$ is finitely generated as a  left $\End_{\bsA}(M_i)$-module.  
For each $1 \leq i \leq n$, take $a_i \in \bsA_{M_i}$.
Then there is an $a\in \bsA$ such that $[a]_{M_i}=a_i$ for each $i$. 
\end{corollary}

\begin{proof}
Set  
$M:=\bigoplus_{i=1}^nM_i$.
Then $M$ is finitely generated as a left $\End_{\bsA}(M)$-module because so is each $M_i$ as a left $\End_{\bsA}(M_i)$-module.
Thus $\bsA_M=\bsB_M$ by Corollary \ref{cor.bicom}.
Here since $M_i$ and $M_j$ are not isomorphic if $i\neq j$, we see that $\End_{\bsA}(M)=\prod_{i=1}^n\End_{\bsA}(M_i)$.
Therefore the endomorphism $(a_i)_i\col M \ra M$ satisfies
$(a_i)_i \in \bsB_M=\bsA_M$. 
\end{proof}

Now let $F$ be a field and $\bsA$ an $F$-algebra.
For an $\bsA$-module $M$ which is  finite-dimensional over $F$, denote by 
$P_{M,a}(X)=\det(X-a\mid M)$ the characteristic polynomial of $a\in \bsA$.
Let $\dim_FM=r$ and set $P_{M,a}(X)=\sum_{i=0}^rc_iX^{r-i}$.
Then the coefficients of $P_{M,a}(X)$ are characterized by $c_i=(-1)^i\Tr_{\wedge^iM}(a) \in F.$
Then we have:

\begin{proposition}\label{prop.BN}
Let $M$ and $M'$ be semi-simple $\bsA$-modules and suppose they have the same finite dimension over $F$.
Set $r:=\dim_F M=\dim_F M'$.
Then the following assertions hold:
\begin{itemize}
\item[$(1)$] Suppose that $F$ has characteristic $p>0$. Then $M$  is isomorphic to $M'$  if and only if $P_{M,a}(X)=P_{M',a}(X)$ for all $a\in \bsA$.
\item[$(2)$] Suppose either $F$ has characteristic $0$ or has characteristic $p>0$ satisfying $r<p$. Then $M$ is isomorphic to $M'$ if and only if $\Tr_{M}(a)=\Tr_{M'}(a)$ for all $a\in \bsA$.
\end{itemize} 
\end{proposition}

\begin{proof}
The ``only if'' parts are obvious, so we prove the ``if'' parts.
Let $\frS$ denote the category of isomorphism classes of simple $\bsA$-modules.
For an object $\mu$ of $\frS$, if $M$ has a simple factor belonging to $\mu$, we denote it by $M_\mu$ and write $m_\mu$ for its multiplicity. 
If $M$ has no simple factors belonging to $\mu$, then we define $M_\mu=\{0\}$ and $m_\mu=0$.
For $M'$, we also define $M'_\mu$ and $m_\mu'$ in the same way.
Then we can take a finite set $H$ of objects of $\frS$ such that 
\[
M=\bigoplus_{\mu \in H}(M_\mu)^{\op m_\mu},\es\es M'=\bigoplus_{\mu \in H}(M_\mu')^{\op m'_\mu},
\]
and either $m_\mu\neq 0$ or $m'_\mu \neq 0$ holds for any $\mu \in H$.
Thus, to prove $M\cong M'$, it is enough to show that $m_\mu=m'_\mu$ for all $\mu \in H$.

(1) Suppose that $P_{M,a}(X)=P_{M',a}(X)$ for all $a\in \bsA$.
For each $\mu \in H$, we fix a simple $\bsA$-module $N_\mu$ which  represents $\mu$.
(Notice that if $m_\mu\neq 0$ and $m'_\mu \neq 0$, then $M_\mu \cong N_\mu \cong M'_\mu$ and 
$P_{M_\mu,a}(X)=P_{N_\mu,a}(X)=P_{M'_\mu,a}(X)$ for all $a\in \bsA$.)
We first claim that $m_\mu \equiv m'_\mu \pmod p$ holds for any $\mu \in H$.
To see this, we fix an arbitrary $\eta \in H$.
Applying Corollary \ref{cor.3} to $\bigoplus_{\mu \in H}N_\mu$ and the sequence $(a_{\mu})_{\mu\in Hz}$ in $\bsA$ with $a_\eta=1$ and $a_\mu=0$ for $\mu\neq \eta$, we can take an element $c \in \bsA$ such that $[c]_{N_\eta}=1$ and $[c]_{N_\mu}=0$ for $\mu \neq \eta$.
Thus 
for any integer $0\leq k\leq r$, we have $\Tr_{\bigwedge^kM}(c)=\Tr_{\bigwedge^k(M_\eta)^{\op m_\eta}}(1)$ and  $\Tr_{\bigwedge^kM'}(c)=\Tr_{\bigwedge^k(M'_\eta)^{\op m'_\eta}}(1)$.
Hence
\[
\Tr_{\bigwedge^k(M_\eta)^{\op m_\eta}}(1)=\Tr_{\bigwedge^k(M'_\eta)^{\op m'_\eta}}(1)
\]
 by assumption.
 Since particularly $\Tr_{(M_\eta)^{\op m_\eta}}(1)=\Tr_{(M'_\eta)^{\op m'_\eta}}(1)$, we have 
 \[
 (m_\eta-m_\eta')\dim_FN_\eta=0 \in F.
 \]
Hence if $\dim_FN_\eta$ is not divisible by $p$, then $m_\eta \equiv m'_\eta \pmod p$.
 On the other hand, suppose that $\dim_FN_\eta$ decomposes as $\dim_FN_\eta=p^\nu d$, where $\nu>0$ and $d$ is not divisible by $p$.
Then of course $p^\nu d \leq r$ and the binomial coefficient $\binom{p^\nu d}{p^\nu}$ is coprime to $p$.
For any integer $m\geq 0$, we see that 
\[
\dim_F\left(\bigwedge^{p^\nu}N_\eta^{\op m}\right)= \dim_F\left(
\bigoplus_{s_1+\cdots+s_m=p^\nu} \left(\bigotimes_{i=1}^m \bigwedge^{s_i}N_\eta  \right)\right)
=\sum_{s_1+\cdots+s_m=p^\nu}\prod_{i=1}^m \binom{p^\nu d}{s_i},
\]
and that $\binom{p^\nu d}{s_i}$ is not divisible by $p$ if and only if $s_i=p^\nu$.
Hence we obtain
\[
\Tr_{\bigwedge^{p^\nu}N_\eta^{\op m}}(1)=\dim_F\left(\bigwedge^{p^\nu}N_\eta^{\op m}\right)=m\binom{p^\nu d}{p^\nu}\ \in F
\]
because $F$ has characteristic $p$.
Hence $\Tr_{\bigwedge^{p^\nu}(M_\eta)^{\op m_\eta}}(1)=\Tr_{\bigwedge^{p^\nu}(M'_\eta)^{\op m'_\eta}}(1)$
implies $(m_\eta -m_\eta')\binom{p^\nu d}{p^\nu}=0\in F$.
 Since $\binom{p^\nu d}{p^\nu} \in F\mal$, we get $m_\eta\equiv m'_\eta \pmod p$.
 
Consequently, there exist semi-simple $\bsA$-modules $L_1, L_1', M_1,$ and $M_1'$ such that 
\[
M=L_1\oplus (M_1)^{\oplus p},\es M'=L'_1\oplus (M'_1)^{\oplus p}, \es \mbox{and}\es L_1 \cong L_1'.
\]
In particular, $\dim_FM_1=\dim_FM_1'$.
Now we have $P_{(M_1)^{\oplus p},a}(X)=(P_{M_1,a}(X))^p$ and $P_{(M'_1)^{\oplus p},a}(X)=(P_{M'_1,a}(X))^p$ for any $a\in\bsA$.
We see that the assumption 
$P_{M,a}(X)=P_{M',a}(X)$ implies $(P_{M_1,a}(X))^p=(P_{M'_1,a}(X))^p$.
Hence $P_{M_1,a}(X)=P_{M_1',a}(X)$.
Repeating the same argument, we can take semi-simple $\bsA$-modules satisfying
\[
M=L_h\oplus (M_h)^{\oplus p^h},\es M'=L'_h\oplus (M'_h)^{\oplus p^h}, \es \mbox{and}\es L_h \cong L_h'
\] 
for any $h$, but  $M$ and $M'$ are finite-dimensional and so  $M_h=\{0\}$ and $M_h'=\{0\}$ for some $h$.
Hence we have $M \cong M'$.

(2) Suppose that $\Tr_{M}(a)=\Tr_{M'}(a)$ for all $a\in \bsA$.
As with the argument in the proof of (1), we have $(m_\eta-m'_\eta)\Tr_{N_\eta}(1)=0$ in $F$ for any $\eta \in H$. 
If $F$ has characteristic $0$, then $m_\eta=m'_\eta$ and so $M \cong M'$.
If $F$ has characteristic $p$ with $r<p$, then we have already seen that  $m_\eta \equiv m'_\eta \pmod p$ for any $\eta\in H$.
On the other hand, we have $|m_\eta-m'_\eta|\leq r<p$.
Hence $m_\eta=m'_\eta$.
\end{proof}

Finally, applying Proposition \ref{prop.BN} to $\bsA=F[G]$, $M=V^\ss$, and $M'=(V')^\ss$, we get Theorem \ref{thm.BN}.

%%%%%%%%%%%%%%%%%%%%%%%%%%%%%%%%%%%%%%%%%%%%%%%%%%%%%%%%%%%%%%%%
% Reference
%
%%%%%%%%%%%%%%%%%%%%%%%%%%%%%%%%%%%%%%%%%%%%%%%%%%%%%%%%%%%

%%%%%%%%%%%%%%%%%%%%%%%%%%%%%%%%%%%%%%%%%%%%%%%%%%%%%%%%%%%%%%%%%%%%%%%%%%%%%%%%%%%%%%%%%%%%%%%%%%%%%%%%%%%%%%%
\vspace{50pt}

Department of Architecture, Faculty of Science and Engineering,

Toyo University

2100, Kujirai, Kawagoe, 
Saitama 350-8585, Japan

{\it  E-mail address} : \email{\tt okumura165@toyo.jp}

{\it URL} : {\tt \url{https://sites.google.com/view/y-okumura/index-e?authuser=0}}
%%%%%%%%%%%%%%%%%%%%%%%%%%%%%%%%%%%%%%%%%%%%%%%%%%%%%%%%%%%%%%%%%%%%%%%%%%%%%%%%%%%%%%%%%%%%%%%%%%%%%%%%%%%%%%%%%%%%%%%%%%%%%%%%%

\end{document}